\documentclass[11pt,a4paper]{article}
\usepackage{multirow}
\usepackage{graphicx}
\usepackage{subfigure}
 \usepackage{epsfig}
 \usepackage{epstopdf}
\usepackage[T1]{fontenc}
\usepackage{extarrows}
\usepackage{geometry}
\usepackage{amsbsy,amsmath,latexsym,amsfonts, epsfig, color, authblk, amssymb, graphics, bm}
\usepackage{epsf,slidesec,epic,eepic}
\usepackage{fancybox}
\usepackage{fancyhdr}
\usepackage{setspace}
\usepackage{nccmath}
\usepackage{extarrows}
\usepackage{multirow}
\usepackage{extarrows}
\usepackage{graphicx}
\usepackage{algorithm}
\usepackage{algorithmic}
\usepackage{color}
\usepackage{setspace}
\usepackage{subfigure}
\usepackage{graphicx}
\usepackage{amsmath}
\usepackage{booktabs}
\usepackage{multirow}
\usepackage{cases}
\usepackage[colorlinks, citecolor=blue]{hyperref}

\newtheorem{theorem}{Theorem}[section]
\newtheorem{lemma}{Lemma}[section]
\newtheorem{definition}{Definition}[section]

\newtheorem{proposition}{Proposition}[section]

\newtheorem{assumption}{Assumption}
\newtheorem{remark}{Remark}[section]

\usepackage[figuresright]{rotating}
\usepackage{adjustbox}
\newenvironment{proof}{{\noindent \bf Proof:}}{\hfill$\Box$\medskip}

\definecolor{lred}{rgb}{1,0.8,0.8}
\definecolor{lblue}{rgb}{0.8,0.8,1}
\definecolor{dred}{rgb}{0.6,0,0}
\definecolor{dblue}{rgb}{0,0,0.5}
\definecolor{dgreen}{rgb}{0,0.5,0.5}

\begin{document}

 \title{A Globalized Semismooth Newton Method for Prox-regular Optimization Problems}

\author{
 Yuqia Wu\footnote{Department of Mathematics Sciences, Shenzhen University, 
 	China (\href{mailto:mayuqiawu@szu.edu.cn}{mayuqiawu@szu.edu.cn}).} ,\ \ 
 Pengcheng Wu\footnote{Department of Mathematics, Soochow University, China (\href{mailto:pcwu0725@163.com}{pcwu0725@163.com}).},\ \
 Yaohua Hu\footnote{Department of Mathematics Sciences, Shenzhen University, 
 	China (\href{mailto:mayhhu@szu.edu.cn}{mayhhu@szu.edu.cn}).},\ \
 Shaohua Pan\footnote{School of Mathematics, South China University of Technology, China (\href{mailto:shhpan@scut.edu.cn}{shhpan@scut.edu.cn}).} ,\ \
 Xiaoqi Yang\footnote{Department of Applied Mathematics, The Hong Kong Polytechnic University, 
 	Hong Kong (\href{mailto:mayangxq@polyu.edu.hk}{mayangxq@polyu.edu.hk}).} .
 }

 \maketitle
 \begin{abstract}
     We are concerned with a class of nonconvex and nonsmooth composite optimization problems, comprising a twice differentiable function and a prox-regular function. We establish a sufficient condition for the proximal mapping of a prox-regular function to be single-valued and locally Lipschitz continuous. By virtue of this property, we propose a hybrid of proximal gradient and semismooth Newton methods for solving these composite optimization problems, which is a globalized semismooth Newton method. The whole sequence is shown to converge to an $L$-stationary point under a Kurdyka-{\L}ojasiewicz exponent assumption. Under an additional error bound condition and some other mild conditions, we prove that the sequence converges to a nonisolated $L$-stationary point at a superlinear convergence rate. 
     Numerical comparison with several existing second order methods reveal that our approach performs comparably well in solving both the $\ell_q(0<q<1)$ quasi-norm regularized problems and the fused zero-norm regularization problems.
 \end{abstract}

 \section{Introduction}\label{sec1.0}
 
 We are interested in the following nonsmooth and nonconvex composite optimization problem 
 \begin{equation}\label{model}
  \min_{x\in\mathbb{R}^n} F(x):= f(x) + g(x).
 \end{equation}
 We assume that $F$ satisfies the following assumptions.

 \begin{assumption}\label{ass-g}
 {\rm (i)} $f\!:\mathbb{R}^n\to\overline{\mathbb{R}}\!:=\mathbb{R}\cup\{\infty\}$ is twice differentiable on an open set $\mathcal{O}\supset {\rm dom}g$, and $\nabla\!f$ and $\nabla^2\! f$ are Lipschitz continuous on $\mathcal{O}$, where the former has a Lipschitz constant $L$.

 \noindent
 {\rm (ii)} $g\!:\mathbb{R}^n\to \overline{\mathbb{R}}$ is proper and lsc, and $g$ has a computable proximal mapping.

 \noindent
 {\rm (iii)} 
 $F$ is lower bounded, i.e., $\underline{F}:=\inf_{x\in\mathbb{R}^n} F(x) >-\infty$.
 \end{assumption}

 Problem \eqref{model} frequently arises from sparse optimization, which has a wide range of applications in signal processing, statistics, machine learning, and so on. The function $g$ in Assumption \ref{ass-g} (ii) includes the $\ell_q$ quasi-norm $\|x\|_q^q$ with $q=\frac{1}{2}, \frac{2}{3}, 1$, the zero norm, SCAD, MCP (see \cite{xu12,xu23,fan2001variable,zhang2010nearly}), and the fused zero-norms \cite{wu23b}. 

 \subsection{Related works}\label{sec1.1}
 
 Recent years have witnessed active research interests in Newton-type methods for the general nonsmooth and nonconvex composite problem \eqref{model}.
 The proximal Newton-type method has been shown to be effective for tackling problem \eqref{model} with a convex or weakly convex $g$ (see, e.g., \cite{bertsekas1982projected, Lee14, Yue19, mordukhovich2023globally, Kanzow21,liu2024inexact}). The semismooth Newton method has been a favored second-order method for solving problem \eqref{model} or its special case, as evidenced in \cite{milzarek2014semismooth,xiao2018regularized, Themelis19,zhou2023revisiting, hu2022local, hu2023projected,ouyang2024trust}. Specifically, Ouyang and Milzarek \cite{ouyang2024trust} considered problem \eqref{model} with convex $g$, and proposed a trust region-type normal map-based semismooth Newton method, which achieves convergence of the sequence under the KL property of a merit function, and a superlinear convergence rate under mild conditions.  In addition, the second-order algorithms designed by the forward-backward envelope (FBE) of $F$ also have become popular ones; see \cite{stella17, Themelis18, Themelis19, themelis21, khanh2023globally, khanh2024coderivative}. 

 When the proximal mapping of $g$ is accessible, a natural residual mapping of \eqref{model} is the one from the fixed point form of the proximal gradient (PG) method, that is,
 \begin{equation}\label{eq-Rx}
 \mathcal{R}_{\gamma}(x)\!:= \gamma^{-1}\left(x-\mathcal{T}_{\gamma}(x)\right) \ \ {\rm with}\ \ \mathcal{T}_{\gamma}(x):= \mathcal{P}_{\gamma}g(x-\gamma \nabla \!f(x))\quad{\rm for}\ x\in\mathbb{R}^n,
 \end{equation}
 where $\mathcal{P}_{\gamma}g$ is the proximal mapping of $g$ associated with parameter $\gamma>0$; see \eqref{def-moreauenvelope} for its definition. To apply the semismooth Newton method to solve the generalized equation $0\in \mathcal{R}_{\gamma}(x)$, the key is the existence of the Clarke Jacobian of $\mathcal{R}_{\gamma}$ at the iterate, at which $\mathcal{R}_{\gamma}$ should be single-valued and locally Lipschitz continuous \cite{qi1993nonsmooth}. When $g$ is convex, $\mathcal{P}_{\gamma}g$ is single-valued and nonexpansive everywhere \cite[Proposition 12.19]{RW09}. By leveraging this property, for problem \eqref{model} with convex $f$ and $g$, Themelis et al. \cite{Themelis19} proposed a globally convergent semismooth Newton method to seek a root of the system $\mathcal{R}_{\gamma}(x)=0$, which achieves superlinear convergence by requiring that $F$ has a locally quadratic growth property at the limit point at which $\mathcal{R}_{\gamma}$ is strict differentiable, and restricting the length of the Newton directions. However, this method cannot be directly extended to solve \eqref{model} with nonconvex $g$ since now the single-valuedness and local Lipschitz continuity of $\mathcal{R}_{\gamma}$ cannot be guaranteed. Even if $g$ is prox-regular at $x$ for some $v \in \partial g(x)$, according to \cite[Proposition 13.37]{RW09}, there only exist a neighborhood $\mathcal{V}(x)$ of $x$ and a constant $\gamma_0>0$ such that $\mathcal{P}_{\gamma}g$ is single-valued and locally Lipschitz continuous at $z+\gamma v$ for all $z\in\mathcal{V}(x)$ and $\gamma\in(0,\gamma_0)$. This guarantees the single-valuedness and local Lipschitz continuity of $\mathcal{P}_{\gamma}g$ in \eqref{eq-Rx} at $x + \gamma v$ instead of $x-\gamma\nabla f(x)$, which prevents semismooth Newton method from solving the nonsmooth system $0\in\mathcal{R}_{\gamma}(x)$. To overcome this difficulty, Hu et al. \cite{hu2023projected} investigated problem \eqref{model} with a so-called strongly prox-regular $g$ to ensure the single-valuedness and local Lipschitz continuity of $\mathcal{R}_{\gamma}$ on ${\rm dom}\,g$. They proposed a globally convergent semismooth Newton method with a superlinear convergence rate under the BD-regularity of the limit point. It is worth mentioning that, a similar variant of the prox-regularity, called $r$-prox-regular, was also used in \cite{bareilles22} to prove the stability of the manifold identification. 

For convex composite problems, Khanh et al. \cite{khanh2023globally} proposed a coderivative-based Newton method, and obtained the superlinear convergence under the semismooth$^*$ property and the twice epi-differentiability of $\partial g$ at the limit. This method was extended to handle structured nonconvex optimization in \cite{khanh2024coderivative}, which seeks a critical point of \eqref{model} by solving the system $0\in\mathcal{R}_{\gamma}(x)$. Since the Newton direction is obtained by coderivative of $\mathcal{R}_{\gamma}$, its single-valuedness and local Lipschitz continuity is unnecessary. They established a superlinear convergence for a prox-regular $g$ by assuming that the limit point is a tilt-stable local minimum and $\partial g$ has the ${\rm semismooth}^*$ property. 
  It was pointed out in \cite[Eq. (2.2)]{gfrerer2021semismooth} that if $\mathcal{R}_{\gamma}$ is single-valued at $x$, then
 $$ {\rm conv} D^*\mathcal{R}_{\gamma}(x)(v)  = \{A^\top v\ | \ A\in  \partial_C\mathcal{R}_{\gamma}(x)\},$$
 where $D^*\mathcal{R}_{\gamma}(x)(v)$ denotes the coderivative of $\mathcal{R}_{\gamma}$ at $x$ for $v$. 
 From this point of view, the solvability of the generalized Newton systems at $(x,v) \in {\rm gph}\mathcal{R}_{\gamma}$ is somehow restricted. In fact, one of the sufficient conditions is the strong metric subregularity of $\mathcal{R}_{\gamma}$ at $(x, v)$ for some $v\in \mathcal{R}_{\gamma}(x)$; see \cite[Lemma 3.1]{khanh2024coderivative}. Recently, Gfrerer \cite{gfrerer2024globally} introduced from another perspective a globalized ${\rm semismooth}^*$ Newton method that does not require the single-valuedness of $\mathcal{R}_{\gamma}$. This method achieves the superlinear convergence under the assumptions of SCD regularity and SCD ${\rm semismooth}^*$ property of $\partial F$ at the origin. These two assumptions are respectively weaker than metric regularity and ${\rm semismooth}^*$ property, but similar to other generalized Newton methods, the well-definedness of the generalized Newton system still requires strong conditions.

 We are interested in designing a globalized semismooth Newton method for \eqref{model} by solving the system $0 \in \mathcal{R}_{\gamma}(x)$ that owns global and superlinear convergence simultaneously, by exploring the sufficient conditions of the existence of Clarke Jacobian of $\mathcal{R}_{\gamma}$. It is worth emphasizing that the function $g$ is not strongly prox-regular as defined in \cite{hu2023projected}, so its proposed algorithm is inapplicable to our problem \eqref{model}. In fact, to ensure the existence of the Clarke Jacobian of the system is a challenging task.
 
 \subsection{Main contribution}\label{sec1.2}
 
 This paper aims at developing a globally convergent semismooth Newton method to seek an $L$-stationary point of problem \eqref{model}, by employing the residual function induced by the proximal mapping of $g$. To guarantee the global performance of semismooth Newton method, following the idea in \cite{Themelis19}, we incorporate the PG to safeguard the iterates generated from semismooth Newton method. Our algorithm is a hybrid of PG and semismooth Newton methods, and will generate two sequences $\{x^k\}_{k\in\mathbb{N}}$ and $\{y^k\}_{k\in\mathbb{N}}$, where the former is from the PG step, and the latter comes from the semismooth Newton step. 

 At the current iterate $y^k$, a PG step with line search is first conducted to produce $x^k\in \mathcal{T}_{\gamma_k}(y^k)$, where $\mathcal{T}$ is defined in \eqref{eq-Rx} and $\gamma_k>0$ is the step-size. Then, we would use a semismooth Newton step to solve 
  \begin{equation}\label{system-Rx}
  0\in \mathcal{R}_{\gamma_k}(x).
 \end{equation}
 Consider that $\mathcal{R}_{\gamma_k}$ may not be single-valued and locally Lipschitz continuous at $x^k$, for which its Clarke generalized Jacobian is ill-defined. For this reason, we search for some $\widetilde{\gamma}_k>0$ to ensure that $\mathcal{R}_{\widetilde{\gamma}_k}$ is single-valued and locally Lipschitz continuous at $x^k$. The existence of such $\widetilde{\gamma}_k$ is due to Proposition \ref{Forder-Fgam}. 
 After that, the iteration enters the semismooth Newton step, which comprises a semismooth Newton step for system \eqref{system-Rx} to obtain $d^k \in \mathbb{R}^n$, and a line search between the segment of $x^k+d^k$ and $\mathcal{T}_{\widetilde{\gamma}_k}(x^k)$ to determine the next iterate $y^{k+1}$. The process is repeated until the termination condition is met.
 
 The main contributions of this work are summarized as follows.
 \begin{itemize}
 \item For a prox-bounded function $h\!:\mathbb{R}^n\to\overline{\mathbb{R}}$, we prove that if $h$ is prox-regular at $\overline{x}$, then for sufficiently small $\gamma>0, $ $\mathcal{P}_{\gamma}h$ is single-valued, locally Lipschitz continuous and monotone around $\overline{x}$ (see Proposition \ref{prop-proxmap}), which not only enhances the result of \cite[Proposition 13.37]{RW09} but also lays the cornerstone for our algorithm.   

 \item By leveraging the above result, we propose a hybrid of PG and semismooth Newton methods for solving problem \eqref{model}. 
 The two generated sequences are proved to converge to an $L$-stationary point, under the assumption that $F$ satisfies the KL property of exponent $\theta\in[\frac{1}{2},\frac{1}{2-\varrho}]$ for $\varrho\in(0,1)$, a stronger assumption than the KL property that is usually used for the establishment of global convergence. If in addition the limit point is a local minimum of $F$ and the residual function is metrically subregular at the limit point, then the distance sequence of $y^k$ from the $L$-stationary point set is shown to have a superlinear convergence rate under some mild conditions. As far as we know, this is the first work to establish the global and superlinear convergence of the iterate sequence generated by the globalized semismooth Newton method for solving the composite problem with nonconvex $g$.

 \item  
 Numerical experiments for the $\ell_q(0<q<1)$ quasi-norm regularized problem and fused zero-norm regularization problem show that the proposed algorithm is comparable with several existing second order methods. 
 \end{itemize}

  \subsection{Notation}
 For any $x\in\mathbb{R}^n$ and $\epsilon>0$, $\mathbb{B}(x,\epsilon):= \{ z\ | \ \|z-x\|\leq \epsilon\}$ denotes the ball centered at $x$ with radius $\epsilon.$ For a closed and convex set $C\subset \mathbb{R}^n$, we denote by $\mathcal{N}_{C}(x)$ the normal cone of $C$ at $x$. For a set $D \subset \mathbb{R}^n$, ${\rm proj}_{D}(z)$ denotes the projection of $z$ onto $D$, and ${\rm conv}(D)$ is the convex hull of $D$. For any $k\in \mathbb{N}$ and $x\in\mathbb{R}^n$, we write $[k]:= \{1,2,..., k\}$, ${\rm supp}(x) = \{i \ | \ x_i \neq 0\}$ and $|x|_{\min} = \min\{
 |x_i|\ | \ i\in{\rm supp}(x)\}$. $I$ and ${\bf 1}$ are the identity matrix and the vector of ones, respectively, whose dimensions are adaptive to the context. Given a symmetric matrix $A$, the spectral norm of $A$ is denoted by $\|A\|_2$, and the smallest eigenvalue of $A$ is written by $\lambda_{\min}(A)$. 
  
 \section{Preliminaries}\label{sec2}
Suppose that $H\!:\mathbb{R}^n\rightarrow\mathbb{R}^m$ is a locally Lipschitz continuous function. According to Rademacher's theorem, $H$ is differentiable almost everywhere. Denote by $D_H$ the set on which $H$ is differentiable. We write the Jacobian matrix of $H$ at $x\in D_H$ by $H'(x)$, and define the $B$-subdifferential as in \cite{pang1993nonsmooth} by
$$\partial_BH(x):=\{V\in\mathbb{R}^n \ | \ \lim_{x^k\rightarrow x, \ x^k\in D_H } H'(x^k) = V\}, $$
and its Clarke generalized Jacobian by $\partial_CH(x)  = {\rm conv}(\partial_BH(x))$.
For a proper function $h:\mathbb{R}^n\rightarrow\overline{\mathbb{R}}$, its proximal mapping $\mathcal{P}_{\gamma} h$ and Moreau envelope $e_{\gamma} h$ associated with a parameter $\gamma>0$ are respectively defined as 
 \begin{equation}\label{def-moreauenvelope}
 \mathcal{P}_{\gamma} h(x)\!:=\mathop{\arg\min}_{z\in\mathbb{R}^n}\Big\{  (2\gamma)^{-1}\|z-x\|^2+h(z)\Big\},\ e_{\gamma} h(x)\!:=\!\inf_{z\in\mathbb{R}^n}\Big\{(2\gamma)^{-1}\|z-x\|^2+h(z)\Big\}.
 \end{equation}
 The function $h$ is said to be \textit{prox-bounded} if there exists $\gamma>0$ such that $e_{\gamma} h(x)>-\infty$ for some $x\in\mathbb{R}^n$ (see \cite[Definition 1.23]{RW09}). The supremum of the set of all such $\gamma$ is the threshold $\gamma_{h}$ of prox-boundedness for $h$. 
 
 For an extended real-valued function $h:\mathbb{R}^n \rightarrow \overline{\mathbb{R}}$ and a point $\overline{x} \in {\rm dom}h$, we denote the regular subdifferential of $h$ at $\overline{x}$ by $\widehat{\partial}h(\overline{x})$, and the general (limiting or Mordukhovich) subdifferential of $h$ at $\overline{x}$ by $\partial h(\overline{x})$, see \cite[Definition 8.3]{RW09}. Now we are in a position to introduce the $L$-stationary point 
 of problem \eqref{model}.

\begin{definition}\label{def-Lspoint}
 A vector $x\in\mathbb{R}^n$ is called an \textit{$L$-stationary point} of \eqref{model} if there exists $\gamma > 0$ such that $0\in \mathcal{R}_{\gamma}(x)$. 
\end{definition}

 If $g$ is assumed to be prox-regular and prox-bounded with threshold $\gamma_g$, from \cite[Lemma 7]{wu23b}, we know that $x$ is an $L$-stationary point of $F$ if and only if $x$ is a critical point of $F$ (i.e., $0\in\partial F(x)$). 
 \subsection{Proximal mapping of prox-regular functions}\label{sec2.2}

A prox-regular function was introduced by Poliquin and Rockafellar \cite{poliquin96}, whose proximal mappings have properties similar to a weakly convex function in a local sense. 
 \begin{definition}\label{def-proxregular}
 A function $h\!:\mathbb{R}^n \rightarrow \overline{\mathbb{R}}$ is said to be prox-regular at  $\overline{x}\in{\rm dom}\,h$ for $\overline{v}\in\partial h(\overline{x})$ if $h$ is locally lsc at $\overline{x}$ with $\overline{v}$, and there exist $\rho \geq 0$ and $\varepsilon > 0$ such that, for $(x,v)\in {\rm gph}\partial h$ with $\|x-\overline{x}\|<\varepsilon,\|v-\overline{v}\|<\varepsilon,h(x)<h(\overline{x})+\varepsilon$, we have
 \[
  h(x') \geq h(x) + v^{\top}(x' -x) - (\rho/2) \|x'-x\|^2\quad\forall \|x'-\overline{x}\|\le\varepsilon.
 \]
 $h$ is prox-regular at $\overline{x}$ if $h$ is prox-regular at $\overline{x}$ for all $v\in \partial h(\overline{x})$.
 \end{definition}

From \cite[Proposition 12.19]{RW09}, for proper convex function $h$, its proximal mapping $\mathcal{P}_\gamma h$ is single-valued and nonexpansive for any $\gamma > 0$. For general prox-bounded function $h$,
 from \cite[Proposition 13.37]{RW09}, if  
$h$ is prox-regular at $\overline{x}$ for $\overline{v}\in\partial h(\overline{x})$, then for all $\gamma>0$ small enough, there is a neighborhood of $\overline{x}+\gamma\overline{v}$ on which $\mathcal{P}_{\gamma} h$ is single-valued, Lipschitz continuous and monotone. This neighborhood depends on $\gamma,\overline{x}$ and $\overline{v}$. Next, by strengthening the prox-regularity of $h$ at $\overline{x}$ for $\overline{v}$ to be that of $h$ at $\overline{x}$, we prove that on a neighborhood of $\overline{x}$, independent of $\overline{v}$, $\mathcal{P}_{\gamma} h$ is also single-valued, Lipschitz continuous and monotone. This result will be employed to establish Proposition \ref{Forder-Fgam}, and its proof requires the following lemma. 
\begin{lemma}\label{uniform}
 Let $h\!:\mathbb{R}^n\to\overline{\mathbb{R}}$ be a proper, lsc and prox-bounded function. Then, $h$ is prox-regular at $\overline{x}$ if and only if for each $r>0$ there exist $\varepsilon>0$ and $\rho\in(0,\gamma_{h})$ such that, when $(x\!,\!v)\in{\rm gph}\,\partial h\ {\rm with}\ \|x\!-\!\overline{x}\|\!<\!\varepsilon,\,{\rm dist}(v,\partial h(\overline{x})\cap\mathbb B(0,r))\!<\!\varepsilon,\,h(x)\!<\!h(\overline{x})\!+\!\varepsilon$,
\begin{equation}\label{equa1-uniform1}
 h(x')>h(x)+\langle v,x'-x\rangle-(2\rho)^{-1}\|x'-x\|^2~~\mbox{for~all}~x'\not=x.
\end{equation} 
\end{lemma}
\begin{proof}
 $\Longrightarrow$. Suppose to the contrary that there exist $\overline{r}>0$ and a sequence $\{((x^k,v^k),z^k)\}_{k\in\mathbb{N}}\subset{\rm gph}\,\partial h\times\mathbb{R}^n$ with $\|x^k-\overline{x}\|<1/k$, ${\rm dist}(v^k,\partial h(\overline{x})\cap\mathbb{B}(0,\overline{r}))<\frac{1}{k}, h(x^k)<h(\overline{x})+\frac{1}{k}$ and $z^k\not=x^k$ such that for all $k\in\mathbb{N}$, 
\begin{equation}\label{uniform3}
 h(z^k)\le h(x^k)+\langle v^k,z^k-x^k\rangle-(k/2)\|z^k-x^k\|^2.
\end{equation}
 From ${\rm dist}(v^k,\partial h(\overline{x})\cap\mathbb{B}(0,\overline{r}))<1/k$, if necessary by taking a subsequence, we can assume that $v^k\to \overline{v}\in\mathbb{B}(0,\overline{r})\cap \partial h(\overline{x})$. 
 Then, from the prox-regularity of $h$ at $\overline{x}$ for $\overline{v}$, the prox-boundedness of $h$ and \cite[Proposition 8.46 (f)]{RW09}, there exist $\varepsilon>0$ and $\overline{\rho}\in(0,\gamma_{h})$ such that, whenever $(x,v)\in{\rm gph}\,\partial h$ with $\|v-\overline{v}\|<\varepsilon,\|x-\overline{x}\|<\varepsilon$ and $h(x)<h(\overline{x})+\varepsilon$, 
 \[
  h(x')\ge h(x)+\langle v,x'-x\rangle-(2\overline{\rho})^{-1}\|x'-x\|^2~~\mbox{for~all}~x'\in\mathbb{R}^n.
 \]
 Noting that $\overline{\rho}/2\in(0,\gamma_{h})$, along with the above inequality, it holds that
 \[
  h(x')>h(x)+\langle v,x'-x\rangle-\overline{\rho}^{-1}\|x'-x\|^2~~\mbox{for~all}~x'\ne x.
 \] 
 Recall that ${\rm gph}\,\partial h\ni (x^k,v^k)\!\to\!(\overline{x},\overline{v})$ and $h(x^k)\!<\!h(\overline{x})\!+\!1/k$ and $z^k\ne x^k$. For all $k$ large enough, we have
 $h(z^k)>h(x^k)+\langle v^k,z^k-x^k\rangle-\overline{\rho}^{-1}\|z^k-x^k\|^2$, which is a contradiction to \eqref{uniform3} for all $k>2/\bar\rho$.  The implication in this direction follows.

\noindent
$\Longleftarrow$. It suffices to consider that $\partial h(\overline{x})\ne\emptyset$. Pick any $\overline{v}\in\partial h(\overline{x})$ and set $r=\|\overline{v}\|$. Consider any $(x,v)\in{\rm gph}\,\partial h$ with $\|x-\overline{x}\|<\varepsilon$ and $h(x)<h(\overline{x})+\varepsilon$ and $\|v-\overline{v}\|<\varepsilon$. Note that ${\rm dist}(v,\partial h(\overline{x})\cap\mathbb B(0,r))\le\|v-\overline{v}\|<\varepsilon$. By Definition \ref{def-proxregular}, $h$ is prox-regular at $\overline{x}$ for $\overline{v}$. By the arbitrariness of $\overline{v}\in\partial h(\overline{x})$, $h$ is prox-regular at $\overline{x}$. 
\end{proof}
\begin{proposition}\label{prop-proxmap}
 Let $h\!:\mathbb R^n\to\overline{\mathbb R}$ be a proper, lsc and prox-bounded function. If $h$ is prox-regular at $\overline{x}\in{\rm dom}\,h$, then for all sufficiently small $\gamma>0$, $\mathcal{P}_{\gamma} h$ is locally single-valued, Lipschitz continuous and monotone around $\overline{x}$. 
\end{proposition}
\begin{proof}
 Let $r\!:=1+4\min_{v\in\partial h(\overline{x})}\|v\|$. By Lemma \ref{uniform}, there exist $\varepsilon>0$ and $\rho\in(0,\gamma_h)$ such that \eqref{equa1-uniform1} holds, whenever $(x,v)\in{\rm gph}\,\partial h$ with $\|x-\overline{x}\|<\varepsilon,{\rm dist}(v,\partial h(\overline{x})\cap\mathbb B(0,r))<\varepsilon$ and $h(x)<h(\overline{x})+\varepsilon$. We complete the proof by the following two steps.

 \noindent
  {\bf Step~1}: There is $\widehat{\gamma}\in(0,\rho/2)$ such that for all $\gamma \in (0,\widehat{\gamma})$, $\mathcal{P}_{\gamma} h$ is single-valued at $\overline{x}$. In doing so, we first claim that for each $\gamma\in(0,{\rho}/{2})$, there exists $\overline{z}_{\gamma}\in\partial h(\overline{x}\!-\!\gamma\overline{z}_{\gamma})$ with $\|\overline{z}_{\gamma}\|\le r$. Fix any $\gamma\in(0,\rho/2)$. Note that $\mathcal{P}_{\gamma} h(\overline{x})\ne\emptyset$ because $\rho\in(0,\gamma_h)$. Pick any $\overline{x}_\gamma\in \mathcal{P}_{\gamma} h(\overline{x})$. By the definition of $\mathcal{P}_{\gamma} h(\overline{x})$, it holds that
 \begin{equation}\label{hxgam-ineq}
  h(\overline{x}_\gamma)+(2\gamma)^{-1}\|\overline{x}_\gamma-\overline{x}\|^2\leq h(\overline{x}).
 \end{equation}
 Let $v\in\partial h(\overline{x})$ be such that $\|v\|=\min_{u\in\partial h(\overline{x})}\|u\|$. Noting that $\|v\|<{r}/{4}$, we have $v\in\partial h(\overline{x})\cap\mathbb B(0,r)$. Using inequality \eqref{equa1-uniform1} with $x'=\overline{x}_{\gamma}$ and $x=\overline{x}$ leads to 
 \[
 h(\overline{x}_\gamma) \geq h(\overline{x})+\langle v,\overline{x}_\gamma-\overline{x}\rangle-(2\rho)^{-1}\|\overline{x}_\gamma-\overline{x}\|^2.
\]
 The above two inequalities imply 
 $\big(\frac{1}{2\gamma}-\frac{1}{2\rho}\big)\|\overline{x}_\gamma-\overline{x}\|^2\leq h(\overline{x})-h(\overline{x}_\gamma)-\frac{1}{2\rho}\|\overline{x}_\gamma-\overline{x}\|^2\leq \langle v,\overline{x}-\overline{x}_\gamma\rangle$, 
 which along with $\gamma\in(0,\rho/2)$ and $\|v\|\le r/4$ leads to $\|\overline{x}_\gamma-\overline{x}\|/\gamma\leq r$. Take $\overline{z}_\gamma:=(\overline{x}-\overline{x}_\gamma)/\gamma$. From $\overline{x}_\gamma\in \mathcal{P}_{\gamma} h(\overline{x})$, we get $\overline{z}_\gamma\in\partial h(\overline{x}_\gamma)=\partial h(\overline{x}\!-\!\gamma\overline{z}_{\gamma})$. 
 
 Next we claim that there exists $\widehat{\gamma} \in(0,\rho/2)$ such that for all $\gamma\in (0,\widehat{\gamma})$, 
\begin{equation}\label{dist-zbargamma}
 \|\overline{x}_{\gamma}-\overline{x}\|<\varepsilon/2\ \ {\rm and}\ \ {\rm dist}(\overline{z}_\gamma,\partial h(\overline{x})\cap\mathbb B(0,r))<\varepsilon/2.
\end{equation}
Indeed, from $\|\overline{x}_\gamma-\overline{x}\|\leq \gamma r$, it follows that $\lim_{\gamma\to 0}\overline{x}_\gamma=\overline{x}$, which along with the lsc of $h$ at $\overline{x}$ and inequality \eqref{hxgam-ineq} implies $\lim_{\gamma\to 0}h(\overline{x}_{\gamma})=h(\overline{x})$. To prove this claim, it suffices to argue that ${\rm dist}(\overline{z}_\gamma,\partial h(\overline{x})\cap \mathbb B(0,r))\to0$ as $\gamma\to0$. If not, there exist $\epsilon_0>0$ and a sequence $\gamma_k\to0$ such that for sufficiently large $k\in\mathbb{N}$, ${\rm dist}(\overline{z}_{\gamma_k},\partial h(\overline{x})\cap\mathbb B(0,r))>\epsilon_0$. 
 For each $k\in\mathbb{N}$, let $\overline{x}_{\gamma_k}:=\overline{x}-\gamma_k\overline{z}_{\gamma_k}$. Recall that $\overline{z}_{\gamma_k}\in\partial h(\overline{x}_{\gamma_k})$ and $\|\overline{z}_{\gamma_{k}}\|\le r$ for each $k\in\mathbb{N}$. By taking a subsequence if necessary, we can assume that $\overline{z}_{\gamma_{k}}\to\overline{z}$ with $\|\overline{z}\|\le r$ as $k\to\infty$. Recall that $\overline{x}_{\gamma_k}\to\overline{x}$ and $h(\overline{x}_{\gamma_k})\to h(\overline{x})$ as $k\to\infty$. Then, $\overline{z}\in\partial h(\overline{x})\cap \mathbb B(0,r)$. 
This implies that ${\rm dist}(\overline{z}_{\gamma_k},\partial h(\overline{x})\cap \mathbb B(0,r))\leq \|\overline{z}_{\gamma_k}-\overline{z}\|\to 0$ as $k\to\infty$, which is a contradiction to ${\rm dist}(\overline{z}_{\gamma_k},\partial h(\overline{x})\cap\mathbb B(0,r))>\epsilon_0$. Thus, the claimed \eqref{dist-zbargamma} holds. 

Now fix any $\gamma\in (0,\widehat{\gamma})$. Pick any $\overline{x}_{\gamma}\in\mathcal{P}_{\gamma}h(\overline{x})$ and let $\overline{z}_{\gamma}:=(\overline{x}-\overline{x}_{\gamma})/\gamma$. Then, $\overline{z}_{\gamma}\in\partial h(\overline{x}_{\gamma})$. Invoking \eqref{equa1-uniform1} with $x=\overline{x}_{\gamma}$ and $v=\overline{z}_{\gamma}$ and using $\gamma\in(0,\rho/2)$ leads to  
\[
  h(x')>h(\overline{x}_{\gamma})+\langle\overline{z}_{\gamma},x'-\overline{x}_{\gamma}\rangle-(2\gamma)^{-1}\|x'-\overline{x}_{\gamma}\|^2\quad\mbox{for~all}~x'\ne\overline{x}_{\gamma}
\]
which, by $\overline{x}_{\gamma}=\overline{x}-\gamma\overline{z}_{\gamma}$ and a suitable rearrangement, can equivalently be written as
\[
 h(x')+(2\gamma)^{-1}\|x'-\overline{x}\|^2>h(\overline{x}_\gamma)+(2\gamma)^{-1}\|\overline{x}_\gamma-\overline{x}\|^2\quad\forall x'\not=\overline{x}_\gamma.
\]
This implies that $\mathcal{P}_\gamma h(\overline{x})=\{\overline{x}_\gamma\}$. Thus, for any $\gamma \in (0,\widehat{\gamma})$, $\mathcal{P}_{\gamma} h$ is single-valued at $\overline{x}$.  

\smallskip

\noindent 
 {\bf Step~2}: For each $\gamma\in(0,\widehat{\gamma})$, there exists $\delta_{\gamma}>0$ such that $\mathcal{P}_{\gamma} h$ is single-valued, Lipschitz continuous and monotone on $\mathbb{B}(\overline{x},\delta_{\gamma})$. 
 Fix any $\gamma \in (0,\widehat{\gamma})$. From \cite[Example 5.23 (b)]{RW09}, there exists $\delta\in(0,\varepsilon)$ such that the set $\mathcal{P}_{\gamma}h(\mathbb{B}(\overline{x},\delta))$ is bounded and the mapping $\mathcal{P}_{\gamma}h$ is outer semicontinuous. 
 For any $x\in\mathbb{B}(\overline{x},\delta)$, choose $x_{\gamma}\in\mathcal{P}_{\gamma}h(x)$. Together with $\mathcal{P}_\gamma h(\overline{x})=\{\overline{x}_\gamma\}$, we have $\lim_{x\to\overline{x}}x_{\gamma}=\overline{x}_{\gamma}$.
 In addition, we have $z_{\gamma,x}:=\frac{x-x_\gamma}{\gamma}\in\partial h(x_{\gamma})$. Recalling that $x_{\gamma}\to\overline{x}_{\gamma}$ as $x\to\overline{x}$, we have  $z_{\gamma,x}\to \frac{\overline{x}-\overline{x}_{\gamma}}{\gamma}=\overline{z}_{\gamma}$ as $x\to\overline{x}$. Consequently, there exists $\delta_{\gamma} \in (0,\min\{\delta,\sqrt{2\gamma\varepsilon}\})$ such that for all $x\in\mathbb{B}(\overline{x},\delta_{\gamma})$,  
\begin{align}\label{temp-hineq1}
 &\|x_\gamma-\overline{x}_\gamma\|<\varepsilon/2,\ \ \|z_{\gamma,x}-\overline{z}_\gamma\|<\varepsilon/2,\qquad\qquad\\
 \label{temp-hineq2}
 h(x_\gamma) \leq h(x_\gamma)&+(2\gamma)^{-1}\|x_\gamma-x\|^2\leq h(\overline{x})+(2\gamma)^{-1}\|\overline{x}-x\|^2<h(\overline{x})+\varepsilon.
\end{align}
The first inequality of \eqref{temp-hineq1} along with \eqref{dist-zbargamma} implies that
$\|x_\gamma-\overline{x}\| \leq \|x_\gamma-\overline{x}_{\gamma}\| + \|\overline{x}_{\gamma}-\overline{x}\| < \varepsilon$; while the second inequality of \eqref{temp-hineq1}, along with ${\rm dist}(\overline{z}_\gamma,\partial h(\overline{x})\cap \mathbb B(0,r))\le{\varepsilon}/{2}$, implies that ${\rm dist}(z_{\gamma,x},\partial h(\overline{x})\cap\mathbb B(0,r))<\varepsilon$. Now using \eqref{temp-hineq2} and $\gamma\in(0,\rho)$ and invoking the above \eqref{equa1-uniform1} with $ x= x_{\gamma}$ and $v = z_{\gamma, x}$ yields that for all $x'\ne x_{\gamma}$,
\[
 h(x')+(2\gamma)^{-1}\|x'-x\|^2>h(x_\gamma)+(2\gamma)^{-1}\|x_\gamma-x\|^2,
\]
which implies $\mathcal{P}_\gamma h(x)=\{x_\gamma\}$. Therefore, $\mathcal{P}_{\gamma}h$ is single-valued on $\mathbb{B}(\overline{x},\delta_{\gamma})$. 

Next we show the Lipschitz continuity and monotonicity of $\mathcal{P}_{\gamma}h$ on $\mathbb{B}(\overline{x},\delta_{\gamma})$.
 Fix any $\gamma\in(0,\widehat{\gamma})$. From the above \eqref{dist-zbargamma}, there exists $\overline{z}_{\gamma}^*\in\partial h(\overline{x})\cap\mathbb B(0,r)$ such that $\|\overline{z}_{\gamma}-\overline{z}_{\gamma}^*\|\le{\varepsilon}/{2}$. Define the multifunction $\Gamma_{\gamma}:\mathbb{R}^n\rightrightarrows\mathbb{R}^n$ by 
\[ 
  \Gamma_{\gamma}(x):=\left\{\begin{array}{cl}
  \!\big\{z\in \partial h(x)\ |\ \|z-\overline{z}_{\gamma}^*\|\le\varepsilon\big\}& {\rm if}\ 
   x\in\mathbb{B}(\overline{x},\varepsilon)\ {\rm with}\ h(x)<h(\overline{x})+\varepsilon,\\
   \emptyset &{\rm otherwise}.
   \end{array}\right.
\]
We claim that $\Gamma_{\gamma}+\rho^{-1}I$ is monotone on ${\rm dom}\ \Gamma_{\gamma}$. Indeed, fix any $x^1,x^2\in\mbox{dom}\  \Gamma_\gamma$. Pick any $z^i\in \Gamma_{\gamma}(x^i)$ for $i=1,2$. Obviously, $\|z^i-\overline{z}_{\gamma}^*\| \leq \varepsilon$ and $z^{i}\in\partial h(x^i)$ for $i=1,2$. Invoking the above \eqref{equa1-uniform1} with $(x',x)=(x^1,x^2)$ and $(x',x)=(x^2,x^1)$ leads to   
\begin{align*}
 &h(x^1)\geq h(x^2)+\langle z^2,x^1-x^2\rangle-\frac{1}{2\rho}\|x^1-x^2\|^2\\
 &h(x^2)\geq h(x^1)+\langle z^1,x^2-x^1\rangle-\frac{1}{2\rho}\|x^2-x^1\|^2.
\end{align*}
 Adding the above two inequalities leads to $\langle z^1+x^1/\rho-(z^2+x^2/\rho),x^1-x^2\rangle\geq0$, so $\Gamma_{\gamma} + \rho^{-1}I$ is monotone on dom $\Gamma_{\gamma}$.  Let $\delta\!:= 1/\gamma -1/\rho$. Then $\Gamma_{\gamma}+\gamma^{-1}I = \Gamma_{\gamma}+\rho^{-1}I+\delta I$ and
 $$ (I+\gamma \Gamma_{\gamma})^{-1} = (\gamma\delta [I+\delta^{-1}(\Gamma_{\gamma} + \rho^{-1}I)])^{-1} = (I+M)^{-1}\circ (\gamma\delta)^{-1}I $$
 with $M=\delta^{-1}(\Gamma_{\gamma}+\rho^{-1}I)$. As $M$ is monotone, we have that $(I+M)^{-1}$ is monotone and nonexpansive by \cite[Theorem 12.12]{RW09}, and that the above equation implies that  $(I+\gamma \Gamma_{\gamma})^{-1}$ is monotone and Lipschitz continuous with constant $(\gamma \delta)^{-1}$ on $\mathbb{B}(\overline{x},\delta_{\gamma})$. 
 Fix any $x\in\mathbb{B}(\overline{x},\delta_{\gamma})$. According to the proof of Step 2, there exists $x_{\gamma}$ such that 
$\mathcal{P}_\gamma h(x)=\{x_\gamma\}$ and $\|x_\gamma-\overline{x}_\gamma\|<\varepsilon/2$, which by the definition of $\mathcal{P}_\gamma h(x)$ implies that 
\[
  0\in\partial h(x_{\gamma})+(x_\gamma-x)/\gamma.
\]
 Recall that $\|x_{\gamma}-\overline{x}\|\le\|x_{\gamma}-\overline{x}_{\gamma}\|+\|\overline{x}_{\gamma}-\overline{x}\|\le\varepsilon$ with $h(x_{\gamma})<h(\overline{x})+\varepsilon$ and $z_{\gamma,x}=(x-x_\gamma)/\gamma$ satisfies $\|z_{\gamma,x}-\overline{z}_{\gamma}^*\|\le\|z_{\gamma,x}-\overline{z}_{\gamma}\|+\|\overline{z}_{\gamma}-\overline{z}_{\gamma}^*\|\le\varepsilon$. We can replace $\partial h(x_{\gamma})$ of the above inclusion with $\Gamma_{\gamma}(x_{\gamma})$, and obtain that $\mathcal{P}_{\gamma}h(x)\subset (I\!+\!\gamma \Gamma_{\gamma})^{-1}(x)$. Recall that $\mathcal{P}_{\gamma}h(x)$ is a singleton and $(I\!+\!\gamma \Gamma_{\gamma})^{-1}(x)$ contains one element at most. Hence, $\mathcal{P}_{\gamma}h(x)=(I\!+\!\gamma \Gamma_{\gamma})^{-1}(x)$. 
 Thus, due to the corresponding properties of $(I\!+\!\gamma \Gamma_{\gamma})^{-1}(x)$, $\mathcal{P}_{\gamma}h$ is Lipschitz continuous and monotone on $\mathbb{B}(\overline{x},\delta_{\gamma})$. 
\end{proof}
 \subsection{Forward-backward envelope of $F$}\label{sec2.3}

 The forward-backward  envelope (FBE) of a general composite function of the form \eqref{model} was initially introduced in \cite{Patrinos13}. 
 The FBE of function $F$ associated with a parameter $\gamma>0$ is defined as 
 \begin{equation}\label{def-fbe}
     \begin{aligned}
         F_{\gamma}(x)&:=\inf_{z\in\mathbb{R}^n}\Big\{\ell_\gamma(z;x):=f(x)+\langle \nabla\!f(x), z-x\rangle + \frac{1}{2\gamma}\|z-x\|^2+ g(z)\Big\}\\
  &=e_{\gamma}g(x-\gamma\nabla\!f(x))+f(x)-\frac{\gamma}{2}\|\nabla\!f(x)\|^2\quad\ \forall x\in\mathcal{O}.
     \end{aligned}
 \end{equation}
 Obviously, $\mathcal{T}_{\gamma}$ defined in \eqref{eq-Rx} is the solution mapping of this minimization problem. According to Assumption \ref{ass-g} (i) and \cite[Proposition 4.3]{Themelis18}, we have the following results.
 \begin{lemma}\label{lemma-PGdescent}
 Assume that $g$ is prox-bounded with threshold $\gamma_g>0.$ For every $\gamma\in(0,\gamma_g)$ and $x\in {\rm dom}g$, it holds that
 \begin{itemize}
 \item[{\rm (i)}] $F_{\gamma}(x)\leq F(x)$; 
 
 \item[{\rm (ii)}] $F(z) \!\leq\! F_{\gamma}(x) \!-\!\frac{1-\gamma L}{2\gamma}\|z \!-\! x\|^2$ for all $z\in \mathcal{T}_{\gamma}(x)$.
 \end{itemize}
 \end{lemma}

 If $g$ is proper, lsc, and prox-bounded with threshold $\gamma_g>0$, we know from  \cite[Theorem 1.25]{RW09} that for every $\gamma\in(0,\gamma_{g})$, the multifunction $\mathcal{P}_{\gamma} g\!:\mathbb{R}^n \rightrightarrows \mathbb{R}^n$ is nonempty and compact-valued, and $e_{\gamma}g:\mathbb{R}^n\rightarrow \mathbb{R}$ is finite-valued and continuous. Along with \eqref{def-fbe} and Assumption \ref{ass-g} (i), for every $\gamma\in(0,\gamma_{g})$, $F_{\gamma}$ is a finite-valued and continuous function on $\mathcal{O}$. Next we make use of the prox-regularity of $g$ and apply Proposition \ref{prop-proxmap} to achieve the continuous differentiability of $F_{\gamma}$ on $\mathcal{O}$.
\begin{proposition}\label{Forder-Fgam}
 Consider any $\overline{x}\in{\rm dom}g$. Assume that $g$ is prox-regular at $\overline{x}$ and prox-bounded with threshold $\gamma_g \geq 1/L$. For all $\gamma>0$ small enough, $\mathcal{T}_{\gamma}$ is locally single-valued and Lipschitz continuous around $\overline{x}$. Consequently, $F_{\gamma}$ is continuously differentiable with $\nabla\!F_{\gamma}(x) =Q_{\gamma}(x)\mathcal{R}_{\gamma}(x)$ for $x$ around $\overline{x}$, where, for any given $\gamma>0$,  $Q_{\gamma}$ is a function defined on $\mathcal{O}$ with $Q_{\gamma}(x):=I-\gamma\nabla^2\!f(x)$. 
\end{proposition}
\begin{proof}
 Let $h(x)\!:= g(x)+\langle\nabla\!f(\overline{x}),x\rangle$ for $x\in\mathbb{R}^n$. As $g$ is prox-regular at $\overline{x}$, it follows by \cite[Exercise 13.35]{RW09} that function $h$ is prox-regular at $\overline{x}$. By Proposition \ref{prop-proxmap}, there exists $\widehat{\gamma}>0$ such that for each $\gamma\in(0,\widehat{\gamma})$, $\mathcal{P}_\gamma h$ is single-valued and Lipschitz continuous on $\mathbb{B}(\overline{x},\delta_{\gamma})\subset \mathcal{O}$ for some $\delta_{\gamma}>0$.
 
 Fix any $\gamma\in(0,\widehat{\gamma})$. Note that $\mathcal{P}_\gamma g(\cdot - \gamma \nabla\! f(\overline{x}))=\mathcal{P}_\gamma h(\cdot)$. Hence, $\mathcal{P}_\gamma g$ is single-valued and Lipschitz continuous on $\mathbb{B}(\overline{x}-\!\gamma\nabla\!f(\overline{x}),\delta_{\gamma})$.
 Next we prove that $\mathcal{T}_{\gamma}$ is single-valued and Lipschitz continuous on $\mathbb{B}(\overline{x},\varepsilon_{\gamma})$ with $\varepsilon_{\gamma}:=\frac{\delta_{\gamma}}{1+\gamma L}$. Indeed, for any $x\in\mathbb{B}(\overline{x},\varepsilon_{\gamma})$, we have $x-\gamma\nabla\!f(x) \in \mathbb{B}(\overline{x}-\gamma\nabla f(\overline{x}),\delta_{\gamma})$ by Assumption \ref{ass-g} (i), which implies that $\mathcal{T}_{\gamma}$ is single-valued and Lipschitz continuous on $\mathbb{B}(\overline{x},\varepsilon_{\gamma})$. 
 
 Recall that $\mathcal{P}_{\gamma}g$ is single-valued on $\mathbb{B}(\overline{x}-\gamma\nabla\!f(\overline{x}),\delta_{\gamma})$.
 From \cite[Example 10.32]{RW09}, we conclude that $e_{\gamma}g$ is strictly differentiable on $\mathbb{B}(\overline{x}-\gamma\nabla\!f(\overline{x}),\delta_{\gamma})$. By \eqref{def-fbe}, $F_{\gamma}$ is continuously differentiable on $\mathbb{B}(\overline{x},\varepsilon_{\gamma})$ and at all $x\in\mathbb{B}(\overline{x},\varepsilon_{\gamma})$, 
 \[
  \nabla F_{\gamma}(x)=\gamma^{-1}(I-\gamma\nabla^2\!f(x))(x-\mathcal{P}_{\gamma}g(x-\!\gamma \nabla f(x))=(I-\gamma\nabla^2\!f(x))\mathcal{R}_{\gamma}(x).
 \]
 Recall that $\mathcal{P}_{\gamma}g$ is also Lipschitz continuous on $\mathbb{B}(\overline{x}-\gamma\nabla\!f(\overline{x}),\delta_{\gamma})$. The above equation along with Assumption \ref{ass-g} (i) implies the local Lipschitz continuity of $\nabla F_{\gamma}$ on $\mathbb{B}(\overline{x},\varepsilon_{\gamma})$. 
 \end{proof}
 
 For given $\overline{x}\in{\rm dom}g$ and $\gamma>0$, if $\mathcal{R}_{\gamma}$ defined by \eqref{eq-Rx} is single-valued and Lipschitz continuous around $\overline{x}$, by Proposition \ref{Forder-Fgam} we have that the Clarke Jacobian of $\nabla F_{\gamma}$ at $\overline{x}$ involves that of $\mathcal{R}_{\gamma}$ around $\overline{x}$. The following lemma provides a characterization for the Clarke Jacobian of $\mathcal{R}_{\gamma}$. 
 \begin{proposition}\label{prop-CJac} 
  Fix $\overline{x}\in {\rm dom}g$ and $\gamma<1/L$. Assume that $g$ is prox-regular and prox-bounded with threshold $\gamma_g \geq 1/L$. If $\mathcal{P}_{\gamma}g$ is locally Lipschitz continuous around $\overline{x}\!-\!\gamma \nabla\!f(\overline{x})$, there exists $\delta>0$ such that for all $x\in \mathbb{B}(\overline{x},\delta)$, $\partial_C \mathcal{P}_{\gamma}g(x-\gamma\nabla\!f(x))$ is a set of symmetric matrices and  $\partial_C\mathcal{R}_{\gamma}(x)$ is a nonempty compact set with 
  \[
   \partial_C \mathcal{R}_{\gamma}(x) = \big\{\gamma^{-1}(I - WQ_{\gamma}(x)) \ | \ W\in \partial_C \mathcal{P}_{\gamma}g(x-\gamma\nabla\!f(x)) \big\},
  \]
  and furthermore, there exists $c_0>0$ such that $\|H\|_2\leq c_0$ for all $H\in \partial_C \mathcal{R}_{\gamma}(x)$.
 \end{proposition}
 \begin{proof}
 As $\mathcal{P}_{\!\gamma}g$ is locally Lipschitz continuous around $\overline{x}-\!\gamma\nabla\!f(\overline{x})$ and $\nabla\! f$ is Lipschitz continuous on ${\rm dom}g$, there exists $\delta>0$ such that $\mathcal{R}_{\gamma}$ is Lipschitz continuous on $\mathbb{B}(\overline{x},\delta)$. By \cite[Theorem 9.62]{RW09}, by shrinking $\delta$ if necessary, for each $x\in\mathbb{B}(\overline{x},\delta)$, $\partial_C \mathcal{R}_{\gamma}(x)$ is nonempty and compact. Fix any $x\in\mathbb{B}(\overline{x},\delta)$. Recall that $\mathcal{R}_{\gamma}(z)=\gamma^{-1}(z-\mathcal{T}_{\gamma}(z))$ for any $z\in\mathbb{R}^n$. By the corollary of \cite[Theorem 2.6.6]{clarke90}, it holds that
 \begin{equation*}
   \partial_C \mathcal{R}_{\gamma}(x)= \gamma^{-1}(I - \partial_C \mathcal{T}_{\gamma}(x)).
 \end{equation*}
 Recall that $\mathcal{T}_{\gamma}(z) = \mathcal{P}_{\gamma}g(z-\gamma\nabla f(z))$ for $z\in\mathbb{B}(\overline{x},\delta)$ and $I\!-\!\gamma\nabla^2\!f(\overline{x})$ is nonsingular. From \cite[Lemma 1]{ChanSun08}, for any $x\in\mathbb{B}(\overline{x},\delta)$ (by shrinking $\delta$ if necessary), it holds that 
 \[
  \partial_C \mathcal{T}_{\gamma}(x)= \partial_C \mathcal{P}_{\gamma} g(x-\gamma\nabla f(x))Q_{\gamma}(x).
 \]
 The above two equations give the expression of $\partial_C \mathcal{R}_{\gamma}(x)$. Let $z = x-\gamma \nabla f(x)$. Since $\mathcal{P}_{\gamma} g$ is locally Lipschitz continuous at $x$, by \cite[Example 10.32]{RW09},  $\mathcal{P}_{\gamma}g(z) = z - \gamma\nabla e_{\gamma}g(z)$. Hence, $\partial_C \mathcal{P}_{\gamma}g(z)$ is the convex hull of $\partial_B[\nabla (\frac{1}{2}\|\cdot\|^2 - \gamma e_{\gamma}g(\cdot))](z)$, which implies that every element of $\partial_C \mathcal{P}_{\gamma}g(z)$ is a symmetric matrix.  
 The last part follows the Lipschitz continuity of $\mathcal{R}_{\gamma}$ on $\mathbb{B}(\overline{x},\delta)$ and  \cite[Proposition 2.6.2 (d)]{clarke90}.  
\end{proof}
 \subsection{KL-property and metric subregularity}\label{sec2.4}

 We first recall the Kurdyka-{\L}ojasiewicz (KL) property of an extended real-valued function, which plays a crucial role in the convergence analysis of first-order methods \cite{Attouch09,Attouch10,Attouch13,Bolte14} and second-order methods  \cite{stella17,Themelis18,themelis21,wu23,wu23b,ouyang2024trust} for nonconvex and nonsmooth optimization problems.
 \begin{definition}\label{KL-Def}
  For every $\eta>0$, denote by $\Upsilon_{\!\eta}$ the set consisting of all continuous concave $\varphi\!:[0,\eta)\to\mathbb{R}_{+}$ that are continuously differentiable on $(0,\eta)$ with $\varphi(0)=0$ and $\varphi'(s)>0$ for all $s\in(0,\eta)$. A proper function $h\!:\mathbb{R}^n\!\to\overline{\mathbb{R}}$ is said to have the KL property at $\overline{x}\in{\rm dom}\,\partial h$ if there exist $\eta\in(0,\infty]$, a neighborhood $\mathcal{U}$ of $\overline{x}$ and a function $\varphi\in\Upsilon_{\!\eta}$ such that for all $x\in\mathcal{U}\cap\big[h(\overline{x})<h<h(\overline{x})+\eta\big]$, 
  \[
   \varphi'(h(x)-h(\overline{x}))\,{\rm dist}(0,\partial h(x))\ge 1.
  \]
  If the function $\varphi$ in the above inequality is chosen as $cs^{1-\theta}$ for some $c>0$ and $\theta\in[0,1)$, then $h$ is said to satisfy the KL property at $\overline{x}$ with exponent $\theta$. 
\end{definition}

 It was shown in \cite[Section 2.2]{liu2024inexact} that the KL property of $h$ with exponent has a close relation with the metrical subregularity of its subdifferential mapping. 
 \begin{definition}\label{Def2.2}
 Let $\mathcal{F}\!:\mathbb{R}^n\rightrightarrows\mathbb{R}^n$ be a multifunction and $(\overline{x},\overline{y})\in{\rm gph}\,\mathcal{F}$, we say that $\mathcal{F}$ is metrically subregular at $\overline{x}$ for $\overline{y}$ if there exist $\kappa>0$ and $\delta>0$ such that for all $x\in\mathbb{B}(\overline{x},\delta)$,
$ {\rm dist}(x,\mathcal{F}^{-1}(\overline{y}))\le\kappa{\rm dist}(\overline{y},\mathcal{F}(x)). $
\end{definition}
 
We discuss the relation between the subregularity of $\partial F_{\gamma}$ and the quadratic growth of $F_{\gamma}$, which will be used in the convergence analysis.
 \begin{lemma}\label{lemma-subregularity-growth}
 Assume that $g$ is prox-bounded with threshold $\gamma_g$. Let $\alpha>0$ and $\gamma\in(0,\min\{\gamma_g,\frac{1}{L+2\alpha}\}]$. Consider any point $\overline{x}\in\mathcal{R}_{\gamma}^{-1}(0)$ around which $\mathcal{R}_{\gamma}$ is single-valued. Then the following assertions hold.
 \begin{itemize}
 \item[{\rm (i)}] $\mathcal{R}_{\gamma}$ is metrically subregular at $\overline{x}$ for $0$ if and only if $\partial F_{\gamma}$ is metrically subregular at $\overline{x}$ for $0$.

 \item[{\rm (ii)}] If $\overline{x}$ is a local minimum of $F_{\gamma}$ and $\partial F_{\gamma}$ is metrically subregular at $\overline{x}$ for $0$, then there exist $\delta>0,\kappa>0$ such that for all $x\in\mathbb{B}(\overline{x},\delta)$,
 $$ F_{\gamma}(x)-F_{\gamma}(\overline{x})\ge\kappa[{\rm dist}(x, \mathcal{R}_{\gamma}^{-1}(0))]^{2}.$$
 \end{itemize}
 \end{lemma}
 \begin{proof}
 Noting that  $\gamma \in (0,\frac{1}{L+2\alpha}]$ and $\|\nabla^2\!f(x)\|_2\leq L$ for every $x\in\mathcal{O}$, we have
\begin{equation}\label{minieigen-Qk}
  \lambda_{\min}(Q_{\gamma}(x)) \geq 1-\gamma \|\nabla^2\!f(x)\|_2 \geq 2\alpha/(L+2\alpha)\quad\forall x\in\mathcal{O}.
 \end{equation}
 By the given assumption and \cite[Exercise 10.7 $\&$ Example 10.32]{RW09}, there exists $\delta_0>0$ such that $e_{\gamma}$ is differentiable on $\mathbb{B}$, so $F_{\gamma}$ is differentiable on $\mathbb{B}(\overline{x},\delta_0)\subset\mathcal{O}$ with $\nabla\!F_{\gamma}(x) = Q_{\gamma}(x)\mathcal{R}_{\gamma}(x)$ for $x\in\mathbb{B}(\overline{x},\delta_0)$. Moreover, from \eqref{minieigen-Qk}, $(\mathcal{R}_{\gamma})^{-1}(0)\cap\mathbb{B}(\overline{x},\delta_0) = (\partial F_{\gamma})^{-1}(0)\cap\mathbb{B}(\overline{x},\delta_0)$. For any $\delta'\in(0,\delta_0/2)$ and $x\in\mathbb{B}(\overline{x},\delta')$, we deduce that
 \begin{align}\label{ineq-Fgam}
 &{\rm dist}(x, (\mathcal{R}_{\gamma})^{-1}(0))={\rm dist}(x, (\mathcal{R}_{\gamma})^{-1}(0)\cap\mathbb{B}(\overline{x},\delta_0))\nonumber\\
 &={\rm dist}(x,(\partial F_{\gamma})^{-1}(0)\cap\mathbb{B}(\overline{x},\delta_0))={\rm dist}(x,(\partial F_{\gamma})^{-1}(0)).
 \end{align} 

 \noindent
 (i)  ``$\Longrightarrow$'': Since $\mathcal{R}_{\gamma}$ is subregular at $\overline{x}$ for $0$, there exist $\delta_1>0$ and $\kappa_1>0$ such that 
 \[
  {\rm dist}(x, (\mathcal{R}_{\gamma})^{-1}(0))\leq \kappa_1{\rm dist}(0,\mathcal{R}_{\gamma}(x))\quad{\rm for\ all}\ x\in\mathbb{B}(\overline{x},\delta_1).
 \]
 Set $\delta=\min\{\delta_0,\delta_1\}/2$. Fix any $x\in\mathbb{B}(\overline{x},\delta)$. From the above inequality, \eqref{minieigen-Qk}, \eqref{ineq-Fgam}, and $\nabla\!F_{\gamma}(x) = Q_{\gamma}(x)\mathcal{R}_{\gamma}(x)$, it immediately follows that 
 \begin{equation*}
  {\rm dist}(x,(\partial F_{\gamma})^{-1}(0))\leq\kappa_1\|Q_{\gamma}(x)^{-1}\nabla\! F_{\gamma}(x)\|\le  \kappa_1[(L/2\alpha)+1]\|\nabla F_{\gamma}(x)\|. 
 \end{equation*}
 Then, by the arbitrariness of $x\in\mathbb{B}(\overline{x},\delta)$, the implication in this direction holds.

  \noindent
 ``$\Longleftarrow$''. As $\partial F_{\gamma}$ is subregular at $\overline{x}$ for $0$, there exist $\overline{\delta}\in(0,\delta_0)$ and $\overline{\kappa}>0$ such that 
 \[
  {\rm dist}(x,(\partial F_{\gamma})^{-1}(0))\leq\overline{\kappa}\,\|\nabla\!F_{\gamma}(x)\|\quad{\rm for\ all}\ x\in\mathbb{B}(\overline{x},\overline{\delta}).
 \]
 By Assumption \ref{ass-g} (i), there exists $\overline{c}>0$ such that $\|Q_{\gamma}(x)\|_2\le\overline{c}$ for all $x\in\mathbb{B}(\overline{x},\overline{\delta})$. 
 Fix any $x\in\mathbb{B}(\overline{x},\overline{\delta})$. From the above inequality, \eqref{ineq-Fgam}, and $\nabla\!F_{\gamma}(x) = Q_{\gamma}(x)\mathcal{R}_{\gamma}(x)$,  
 \begin{equation*}
  {\rm dist}(x, (\mathcal{R}_{\gamma})^{-1}(0))\leq\overline{\kappa}\|Q_{\gamma}(x)\mathcal{R}_{\gamma}(x)\| \leq \overline{c}\,\overline{\kappa}\|\mathcal{R}_{\gamma}(x)\|=\overline{c}\,\overline{\kappa}{\rm dist}(0,\mathcal{R}_{\gamma}(x)). 
 \end{equation*}
 Then, by the arbitrariness of $x\in\mathbb{B}(\overline{x},\overline{\delta})$, the implication in this direction follows.

 \noindent
 (ii) By part (i) and \cite[Remark 2.2 (iii)]{artacho2013metric}, there exist $\delta\in(0,\delta_0)$ and $\kappa>0$ such that for all $x\in\mathbb{B}(\overline{x},\delta)$, 
 $ F_{\gamma}(x) - F_{\gamma}(\overline{x}) \geq \kappa[{\rm dist}(x,(\partial F_{\gamma})^{-1}(0))]^2=\kappa[{\rm dist}(x,\mathcal{R}_{\gamma}^{-1}(0))]^2$, where the equality is due to the above \eqref{ineq-Fgam}. The proof is completed.
 \end{proof}

 To end this section, we present an auxiliary lemma.
 \begin{lemma}\label{prop-supp-Cx}
 Let $C\in\mathbb{R}^{l\times n}$ and $\nu >0$. For any $x,y\in\mathbb{R}^n$ with $|Cx|_{\min} \geq \nu$, $|Cy|_{\min} \geq \nu$ and $\|x-y\| < \frac{\nu}{\|C\|_2}$, we have ${\rm supp}(Cx) = {\rm supp}(Cy)$.
\end{lemma}
\begin{proof}
    For $j \in {\rm supp}(Cx)$,
    $|Cy|_j \geq |Cx|_j - |Cx-Cy|_j \geq |Cx|_{\min} - \|Cx-Cy\| > \nu - \nu = 0,$
    which implies that $j\in {\rm supp}(Cy)$ and hence ${\rm supp}(Cx) \subset {\rm supp}(Cy)$. The inverse inclusion holds. The proof is completed. 
\end{proof}
 \section{Globalized semismooth Newton method}\label{sec3}

 To begin with, we make assumptions that $g$ is prox-regular and prox-bounded to ensure the continuous differentiability of $F_{\gamma}$ with small $\gamma$.
 \begin{assumption}\label{ass-prox-regular}
     $g$ is prox-regular over ${\rm dom}g$ and prox-bounded with threshold $\gamma_g \geq 1/L$, where $L$ is the Lipschitz constant of $\nabla f$ on $\mathcal{O}$.
 \end{assumption}
 
  It is known from \cite{ochs2018local,wu23b} that almost all the sparsity-induced functions are prox-regular and prox-bounded. For any $\gamma>0$ and $x\in\mathcal{O}$ such that $\mathcal{P}_{\gamma}g$ is single-valued and locally Lipschitz continuous around $x-\!\gamma\nabla\!f(x)$ and $Q_{\gamma}(x)$ is nonsingular, define
 \begin{equation}\label{GJac-Fgam}
  \partial^2 F_{\gamma}(x)\!:=\big\{\gamma^{-1} Q_{\gamma}(x)(I - WQ_{\gamma}(x))\ |\ W \in \partial_C\mathcal{P}_{\gamma}g(x-\gamma \nabla\!f(x))\big\}.
 \end{equation}
  It is worth noting that this definition follows the one in \cite{Themelis19}, by omitting the term involved the third derivative of $f$ in $\partial_C (\nabla F_{\gamma})(x)$.
 For such $\gamma$ and $x$, by Proposition \ref{prop-CJac}, every element of $\partial^2 F_{\gamma}(x)$ is symmetric. It is worth mentioning that $\partial^2 F_{\gamma}(x)$ is an approximation to $\partial_C(\nabla F_{\gamma})(x)$. Indeed, if $f$ is quadratic, $\partial^2 F_{\gamma}(x)=\partial_C(\nabla F_{\gamma})(x)$ on the set where $F_{\gamma}$ is differentiable and $\nabla F_{\gamma}$ is locally Lipschitz continuous.  In general, as $\mathcal{R}_{\gamma}(x)$ is a residual function, $\partial^2 F_{\gamma}(x)$ is a good approximation to $\partial_C(\nabla F_{\gamma})(x)$. The iteration steps of our algorithm are described as follows.
 \begin{algorithm}[!ht]
 \caption{(a hybrid of PG and semismooth Newton method (PGSSN))}\label{hybrid}
 \textbf{Input:} $\epsilon\geq 0,\beta\in(0, 1),\alpha>0,\overline{\gamma}\ge\frac{1}{L+\alpha},\widetilde{L}=L+2\alpha$, $\tau\in (0,1)$ and $\varrho\in (0,1-\tau)$, $0<\underline{\varsigma}<\overline{\varsigma}$ and $0<\underline{\sigma}<\overline{\sigma}$. Set $k:=0$ and select an initial point $y^0\in\mathbb{R}^n$. 
	
\medskip
{\bf PG step:}

    (1a)
    Seek $\gamma_k \in [\frac{1}{L+\alpha}, \overline{\gamma}]$ and $x^{k} \in \mathcal{T}_{\gamma_{k}}(y^k)$ such that $F(x^{k}) \leq F_{\gamma_k}(y^k) - \frac{\alpha}{2} \|x^k - y^k\|^2$.

    (1b) If $\gamma_k^{-1}\|x^k - y^k\| \leq \epsilon$, then output $y^k$; else go to step (2a). 
	
    \medskip
    {\bf Semismooth Newton step:}

    (2a) Set $\widetilde{\gamma}_k\!:= \widetilde{L}^{-1}\beta^{m_k}$ where
    $m_k$ is the smallest nonnegative integer $m$ such that
    \begin{equation*}
     \mathcal{T}_{\widetilde{L}^{-1}\beta^{m}}\ {\rm is \ locally \ single \ valued\  and\ Lipschitz \ continuous\  at }\ x^k. 
    \end{equation*}

  (2b) Let $\mu_k\!:=\|\mathcal{R}_{\widetilde{\gamma}_k}(x^k)\|^{\tau}$.  Pick $H_k\in\partial^2 F_{\widetilde{\gamma}_k}(x^k),\sigma_k \in [\underline{\sigma},\overline{\sigma}]$ and $\varsigma_k \in [\underline{\varsigma},\overline{\varsigma}]$. 
  Compute
 \begin{align*}
 d^k&\in\mathop{\arg\min}_{d\in\mathbb{R}^n}\frac{1}{2}d^{\top}\big(H_k+\mu_k\sigma_k\big)d + (Q_{\widetilde{\gamma}_k}(x^k)\mathcal{R}_{\widetilde{\gamma}_k}(x^k))^{\top}d  \quad {\rm s.t.}\ \ \|d\|\leq \varsigma_k \|\mathcal{R}_{\widetilde{\gamma}_k}(x^k)\|^{\varrho}.
 \end{align*} 
 	
 (2c) Let $l_k$ be the smallest nonnegative integer $l$ such that
 \begin{equation*}
  F_{\widetilde{\gamma}_k}\big(\beta^{l}(x^k\!+d^k) + (1-\beta^{l})\mathcal{T}_{\widetilde{\gamma}_k}(x^k)\big) \leq F_{\widetilde{\gamma}_{k}}(x^{k}) - \frac{\widetilde{\gamma}_{k}-\widetilde{\gamma}_{k}^2L}{4} \|\mathcal{R}_{\widetilde{\gamma}_{k}}(x^k)\|^2, 
    \end{equation*}
 \qquad and set $y^{k+1}\!:=\beta^{l_k}(x^k\!+d^k) + (1-\beta^{l_k})\mathcal{T}_{\widetilde{\gamma}_k}(x^k)$. Let $k \gets k+1$ and go to (1a).
\end{algorithm}	
\begin{remark}\label{remark-hybrid}
 {\rm (i)} As will be shown in Lemma \ref{well-defined} below, Algorithm \ref{hybrid} is well defined, which produces two iterate sequences $\{x^k\}_{k\in\mathbb{N}}\subset {\rm dom}g$ and $\{y^k\}_{k\in\mathbb{N}}\subset \mathcal{O}$. Among others, $\{x^k\}_{k\in\mathbb{N}}$ is generated by the PG step, and $\{y^k\}_{k\in\mathbb{N}}$ is produced from the semismooth Newton step for  $0\in \mathcal{R}_{\widetilde{\gamma}_{k-1}}(x)$. 
 It is worth noting that $\{y^{k}\}_{k\in\mathbb{N}}$ is not necessarily feasible, while $\{x^k\}_{k\in\mathbb{N}}$ is feasible.

 \noindent
 {\rm  (ii)} In step (2a), we search for $\widetilde{\gamma}_k>0$ such that $\mathcal{T}_{\widetilde{\gamma}_k}$ is single-valued and locally Lipschitz continuous at $x^k$, which guarantees that the Clarke Jacobian of $\mathcal{R}_{\widetilde{\gamma}_k}$ at $x^k$ is nonempty. For some specific $g$ satisfying Assumption \ref{ass-prox-regular}, numerical methods can be used to check whether $\mathcal{T}_{\widetilde{L}^{-1}\beta^m}$ is single-valued and locally Lipschitz continuous at $x^k$. 

 \noindent
 {\rm  (iii)} Now we take a closer look at step (2b). The common semismooth Newton step for minimizing $F_{\widetilde{\gamma}_k}$ is to solve the linear system $H_kd=\nabla F_{\widetilde{\gamma}_k}(x^k) = Q_{\widetilde{\gamma}_k}(x^k)\mathcal{R}_{\widetilde{\gamma}_k}(x^k)$ with $H_k\in\partial_{C}(\nabla F_{\widetilde{\gamma}_k})(x^k)$. Observe that $\partial_{C}(\nabla F_{\widetilde{\gamma}_k})(x^k)$ involves third-time derivative of $f$. We favor $H_k\in\partial^2 F_{\widetilde{\gamma}_k}(x^k)$ which also captures the second-order information of $F_{\widetilde{\gamma}_k}$ well. Note that the subproblem in step (2b) to determine the semismooth Newton direction has a nonempty and compact feasible set, and its objective function is continuous, for which the direction $d^k$ in step (2b)  is well defined. 
 \end{remark}
 \begin{lemma}\label{well-defined}
  Under Assumptions \ref{ass-g}-\ref{ass-prox-regular}, Algorithm \ref{hybrid} is well-defined.
 \end{lemma}
 \begin{proof}
 To show that Algorithm \ref{hybrid} is well defined, it suffices to argue that
 \begin{itemize}
 \item[(i)] there is $\gamma \in [\frac{1}{L+\alpha},\overline{\gamma}]$ such that $F(x)\leq F_{\gamma}(y^k) - \frac{\alpha}{2}\|x-y^k\|^2$ for all $x \in \mathcal{T}_{\gamma}(y^k)$; 

 \item[(ii)] there exists $m_k \in \mathbb{N}$ such that step (2a) holds;

 \item[(iii)] there exists $l_k\in \mathbb{N}$ such that step (2c) holds.
 \end{itemize}
 
 For part (i), according to Lemma \ref{lemma-PGdescent} (ii), for each $\gamma\in(0,\gamma_g)$ and $x \in \mathcal{T}_{\gamma}(y^k)$, we have
 $F(x) \leq F_{\gamma}(y^k) - \frac{1-\gamma L}{2\gamma}\|x - y^k\|^2.$
 Substituting $\gamma=\frac{1}{L+\alpha}$ into the above inequality proves the existence of the desired $\gamma$. Part (ii) directly follows by Assumption \ref{ass-prox-regular} and Proposition \ref{Forder-Fgam}. For part (iii),  if $\mathcal{R}_{\widetilde{\gamma}_k}(x^k) = 0$, from step (2b) we know that $d^k = 0$ and step (2c) directly holds with $l_k = 0$. For the other case,   by noting that $\{\mathcal{T}_{\widetilde{\gamma}_k}(x^k)\}_{k\in\mathbb{N}}\subset{\rm dom}g$, from Lemma \ref{lemma-PGdescent} (i), $F_{\widetilde{\gamma}_k}(\mathcal{T}_{\widetilde{\gamma}_k}(x^k)) \leq F(\mathcal{T}_{\widetilde{\gamma}_k}(x^k))$, which by Lemma \ref{lemma-PGdescent} (ii) yields that for $k\in\mathbb{N}$, 
 \begin{align*}
  F_{\widetilde{\gamma}_k}(\mathcal{T}_{\widetilde{\gamma}_k}(x^k)) \! \leq \! F_{\widetilde{\gamma}_k}(x^k) \!-\! \frac{1-\widetilde{\gamma}_k L}{2\widetilde{\gamma}_k}\|\mathcal{T}_{\widetilde{\gamma}_k}(x^k) - x^k\|^2 \!=\! F_{\widetilde{\gamma}_k}(x^k) - \frac{\widetilde{\gamma}_k\!-\!\widetilde{\gamma}_k^2 L}{2}\|\mathcal{R}_{\widetilde{\gamma}_k}(x^k)\|^2.
 \end{align*} 
 From \cite[Theorem 1.25]{RW09}, $F_{\widetilde{\gamma}_k}$ is continuous at $\mathcal{T}_{\widetilde{\gamma}_k}(x^k)$, which along with $\beta\in(0,1)$ implies that $F_{\widetilde{\gamma}_k}(\beta^l(x^{k}\!+\!d^k) + (1\!-\!\beta^{l}) \mathcal{T}_{\widetilde{\gamma}_k}(x^k)) \rightarrow F_{\widetilde{\gamma}_k}(\mathcal{T}_{\widetilde{\gamma}_k}(x^k))$ as $l\to\infty$. Together with the above inequality and $\widetilde{\gamma}_k-\widetilde{\gamma}_k^2 L>0$, we conclude that step (2c) must hold after a finite number of searches. The proof is completed. 
\end{proof}

 In the rest of this section, $\{x^k\}_{k\in\mathbb{N}}$ and $\{y^k\}_{k\in\mathbb{N}}$ denote the sequences generated by Algorithm \ref{hybrid} with $\epsilon=0$, and we assume that the algorithm will not stop within a finite number of iterations. 
 Next we provide some desirable properties of the iterate and objective sequences, and Proposition \ref{prop-descent} proves the convergence of $\{F(x^k)\}_{k\in\mathbb{N}}$.

 \begin{proposition}\label{prop-descent}
 Assume that Assumptions \ref{ass-g}-\ref{ass-prox-regular} hold. Then we have $$F(x^{k+1})\le F(x^k)-\frac{\alpha}{2}\|x^{k+1}- y^{k+1}\|^2,$$ so $\{F(x^k)\}_{k\in\mathbb{N}}$ is convergent with limit, denoted by $\overline{F}$, and $\lim_{k\to\infty}\|x^k-y^k\|=0$.
 \end{proposition}
 \begin{proof}
 Fix $k\in\mathbb{N}$. By the definitions of $\gamma_{k+1}$ and $\widetilde{\gamma}_k$, $\gamma_{k+1} \geq \frac{1}{L+\alpha} > \widetilde{\gamma}_k$. Pick $\widetilde{z}^{k+1} \in \mathcal{T}_{\widetilde{\gamma}_{k}}(y^{k+1})$. Together with the definition of $\ell$ in \eqref{def-fbe}, it holds that
 \begin{equation}\label{FBEcomp}
  F_{\gamma_{k+1}}(y^{k+1}) \le \ell_{\gamma_{k+1}}(\widetilde{z}^{k+1};y^{k+1}) \le \ell_{\widetilde{\gamma}_k}(\widetilde{z}^{k+1};y^{k+1}) = F_{\widetilde{\gamma}_k}(y^{k+1}),
 \end{equation}
 which along with Lemma \ref{lemma-PGdescent} implies that 
 \begin{align}\label{descent-K21}
  F(x^{k+1}) & \leq F_{\gamma_{k+1}}(y^{k+1})- \frac{\alpha}{2} \|x^{k+1} - y^{k+1}\|^2 \leq F_{\widetilde{\gamma}_{k}}(y^{k+1})- \frac{\alpha}{2} \|x^{k+1} - y^{k+1}\|^2 \nonumber\\
  & \leq F_{\widetilde{\gamma}_k}(x^k) - \frac{\alpha}{2} \|x^{k+1} - y^{k+1}\|^2 \leq F(x^k)  - \frac{\alpha}{2} \|x^{k+1} - y^{k+1}\|^2,
 \end{align}
 where the third inequality is due to step (2c) in Algorithm \ref{hybrid}. Thus, for any $l\in\mathbb{N}$, we have $\frac{\alpha}{2}\sum_{k=0}^{l}\|x^{k+1}-y^{k+1}\|^2\le F(x^0)-F(x^{l+1})$. Passing the limit $l\to\infty$ and using the lower boundedness of $F$ in Assumption \ref{ass-g} (iii) leads to $\lim_{k\to\infty}\|x^k-y^k\|=0$. 
\end{proof}

To conduct further convergence analysis, we make the following assumption.
 \begin{assumption}\label{ass-bounded}
     The sequence $\{x^k\}_{k\in\mathbb{N}}$ is bounded.
 \end{assumption}
 
We remark here that this assumption can be guaranteed if the objective function is level bounded, by noting that $\{F(x^k)\}_{k\in\mathbb{N}}$ is a descent sequence by Proposition \ref{prop-descent}. Next, we prove the boundedness of $\{y^k\}_{k\in\mathbb{N}}$, $\{\gamma_k\}_{k\in\mathbb{N}}$ and $\{\widetilde{\gamma}_k\}_{k\in\mathbb{N}}$ under Assumption \ref{ass-bounded}, as well as the descent property of $\{F_{\widetilde{\gamma}_k}(x^k)\}_{k\in\mathbb{N}}$.
 
\begin{lemma}\label{lemma-bounded}
   Under Assumptions \ref{ass-g}-\ref{ass-bounded}, the following results hold.
   \begin{itemize}
       \item[{\rm (i)}] There exists a compact set $\Delta$ containing both $\{x^k\}_{k\in\mathbb{N}}$ and $\{y^k\}_{k\in\mathbb{N}}$. 
 
 \item[{\rm (ii)}] There is $\gamma_{\min}>0$ such that $\{\widetilde{\gamma}_k\}_{k\in\mathbb{N}}\!\subset[\gamma_{\min}, \widetilde{L}^{-1}]$ and $\{\gamma_k\}_{k\in\mathbb{N}}\subset\![\gamma_{\min},\overline{\gamma}]$. 

 \item[{\rm (iii)}]  There exists $\overline{c}_1>0$ (only depending on $\gamma_{\rm min}$ and $L$) such that for all $k\in\mathbb{N}$, 
 $\max\big\{F(x^{k+1}) - F(x^k),F_{\widetilde{\gamma}_{k+1}}(x^{k+1}) - F_{\widetilde{\gamma}_{k}}(x^k)\big\}\leq -\overline{c}_1\|\mathcal{R}_{\widetilde{\gamma}_{k}}(x^k)\|^2.$
   \end{itemize} 
\end{lemma}
\begin{proof}
(i) From Assumption \ref{ass-bounded}, $\{x^k\}_{k\in\mathbb{N}}$ is a bounded sequence. Moreover, from Proposition \ref{prop-descent}, we have for all $k\in\mathbb{N}$, $\|x^{k} - y^k\| \leq \sqrt{2\alpha^{-1}(F(x^{k-1})-F(x^k))} \leq \sqrt{2\alpha^{-1}(F(x^0)-\underline{F})},$
 which implies that $\{y^k\}_{k\in\mathbb{N}}$ is also bounded. Then, there exists a compact set $\Delta\subset\mathbb{R}^n$ such that both sequences are contained in $\Delta$.

 \noindent
 (ii) We claim that there exists $\gamma_{\min}>0$ such that $\widetilde{\gamma}_k \in [\gamma_{\min},\widetilde{L}^{-1}]$ for every $k\in\mathbb{N}$. By Proposition \ref{Forder-Fgam}, for each $x\in \Delta$, where $\Delta$ is the set in (i), there exist $m_x \in \mathbb{N}$ and $\epsilon_x > 0$ such that $\mathbb{B}(x,\epsilon_x)\subset\mathcal{O}$ and $\mathcal{T}_{\widetilde{L}^{-1}\beta^{m_x}}$ is single-valued and Lipschitz continuous on $\mathbb{B}(x,\epsilon_x)$. Note that $\Delta \subset \bigcup_{x\in\Delta}\mathbb{B}^{\circ}(x,\epsilon_x)$. From Heine-Borel theorem, there exist $x_1,\ldots,x_q\in\Delta$ such that $\Delta \subset \bigcup_{i\in [q]} \mathbb{B}^{\circ}(x_i, \epsilon_{x_i})$. Let $m_{\max}\!:=\max\{m_{x_1},\ldots,m_{x_q}\}$. Then, for each $x\in\Delta$, there exists $m\in\{0,1,\ldots,m_{\max}\}$ such that $\mathcal{T}_{\widetilde{L}^{-1}\beta^m}$ is single-valued and locally Lipschitz continuous at $x$. Recall that $\{x^k\}_{k\in\mathbb{N}} \subset \Delta$. By the definition of $\widetilde{\gamma}_k$, we have $\widetilde{\gamma}_k\ge \widetilde{L}^{-1}\beta^{m_{\max}}:=\gamma_{\min}$. In addition, by step (2b) of Algorithm \ref{hybrid}, $\widetilde{\gamma}_k\leq \widetilde{L}^{-1}$. Thus, the claimed conclusion holds. Now from the PG step and $\gamma_{\min}\leq \widetilde{L}^{-1}<\frac{1}{L+\alpha}$, we immediately have $\{\gamma_k\}_{k\in\mathbb{N}} \subset [\gamma_{\min},\overline{\gamma}]$.

 \noindent
(iii) Fix any $k\in\mathbb{N}$. From the above \eqref{descent-K21}, we have $F(x^{k+1})\leq F_{\widetilde{\gamma}_{k}}(y^{k+1})$, which together with step (3b) and Lemma \ref{lemma-PGdescent} implies that 
  \begin{align*}
  F_{\widetilde{\gamma}_{k+1}}(x^{k+1})\leq F(x^{k+1}) \leq F_{\widetilde{\gamma}_{k}}(y^{k+1})
  &\leq F_{\widetilde{\gamma}_k}(x^k)-\frac{1}{4}(\widetilde{\gamma}_{k} -\widetilde{\gamma}_{k}^2L)\|\mathcal{R}_{\widetilde{\gamma}_{k}}(x^k)\|^2\\
  &\le F(x^{k})-\frac{1}{4}(\widetilde{\gamma}_{k} -\widetilde{\gamma}_{k}^2L)\|\mathcal{R}_{\widetilde{\gamma}_{k}}(x^k)\|^2.
 \end{align*} 
 By part (ii), we have $\widetilde{\gamma}_k\in [\gamma_{\min}, \widetilde{L}^{-1}]$.  Let $\overline{c}_1:=\frac{1}{4} \min\{\gamma_{\min}-\!L\gamma_{\min}^2, \widetilde{L}^{-1}-L\widetilde{L}^{-2}\}.$ We obtain $\widetilde{\gamma}_k - L \widetilde{\gamma}_k^2 \geq 4\overline{c}_1$, and hence the desired result holds.
\end{proof}
 \begin{proposition}\label{prop-dkzero}
  Under Assumptions \ref{ass-g}-\ref{ass-bounded}, the following two statements  hold.
 \begin{itemize}
 \item[{\rm (i)}]   $\lim_{k\rightarrow \infty} \|\mathcal{R}_{\widetilde{\gamma}_k}(x^k)\| = 0=\lim_{k \rightarrow \infty} \|d^k\|$.
 
 \item[{\rm (ii)}]  $\lim_{k\rightarrow\infty} \|x^{k+1} - x^k\|= 0.$
    \end{itemize}  
\end{proposition}
\begin{proof}
(i)  From Lemma \ref{lemma-bounded} (iii),
\(
 \|\mathcal{R}_{\widetilde{\gamma}_k}(x^{k})\|^2 \le\overline{c}_1^{-1}(F(x^{k}) - F(x^{k+1})).
\)
By observing that $\{F(x^k)\}_{k\in\mathbb{N}}$ is a convergent sequence from Proposition \ref{prop-descent}, we have that $\lim_{k\rightarrow \infty} \|\mathcal{R}_{\widetilde{\gamma}_k}(x^{k})\| = 0$. From step (2b), 
$ \|d^k\|\leq \varsigma_k\|\mathcal{R}_{\widetilde{\gamma}_k}(x^k)\|^{\varrho} \leq  \overline{\varsigma}\|\mathcal{R}_{\widetilde{\gamma}_k}(x^k)\|^{\varrho}, $
which by $\lim_{k\rightarrow \infty} \|\mathcal{R}_{\widetilde{\gamma}_k}(x^k)\|= 0$ and $\varrho\in (0,1)$ implies that $\lim_{k\rightarrow\infty} \|d^k\| = 0$. 
        
\noindent
(ii) For every $k\in\mathbb{N}$,
$\|x^{k+1} - x^k\| \leq \|x^{k+1} - y^{k+1} \| + \|y^{k+1} -x^k\|\leq \|x^{k+1} - y^{k+1} \| + \|d^k\| + \widetilde{\gamma}_k \|\mathcal{R}_{\widetilde{\gamma}_k}(x^k)\|$, which by part (i), Proposition \ref{prop-descent} and Lemma \ref{lemma-bounded} (ii) implies that $\lim_{k\rightarrow \infty}\|x^{k+1} - x^k\|= 0$. 
\end{proof} 
 \section{Global convergence}\label{sec4}
 
 Based on the results established in the previous section, we are ready to establish the subsequence convergence of $\{x^k\}_{k\in\mathbb{N}}$, i.e., to prove that its every cluster point is an $L$-stationary point of \eqref{model}. We denote by $\omega(x^0)$ the set of cluster points of $\{x^k\}_{k\in\mathbb{N}}$ starting from $x^0$.
\begin{proposition}\label{prop-supp}
 Under Assumptions \ref{ass-g}-\ref{ass-bounded}, the following results hold. 
 \begin{itemize} 
 \item[{\rm (i)}] For every $x^*\in\omega(x^0)$, it holds that  $0\in \mathcal{R}_{\gamma}(x^*)$ for some $\gamma \in [\frac{1}{L+\alpha},\overline{\gamma}]$.

 \item[{\rm (ii)}] For every $x^*\in\omega(x^0)$, there exists $\epsilon_{x^*} > 0$ such that $\mathbb{B}(x^*,\epsilon_{x^*})\subset\mathcal{O}$ and 
 $\mathcal{R}_{\widetilde{L}^{-1}}$ is locally Lipschitz continuous on $\mathbb{B}(x^*,\epsilon_{x^*})$.

 \item[{\rm (iii)}] There exists $\overline{k}\in\mathbb{N}$ such that $\widetilde{\gamma}_k = \widetilde{L}^{-1}$ for $k>\overline{k}$. 

 \item[{\rm (iv)}] For every $x^*\in\omega(x^0)$, it holds that $F(x^*)\!=\!\overline{F}\!=\!F_{\widetilde{L}^{-1}}(x^*)\!=\!\lim_{k\to\infty}F_{\widetilde{L}^{-1}}(x^k)$. 
\end{itemize}
\end{proposition}
\begin{proof}
  (i) Pick any $x^*\in\omega(x^0)$. There exists $K \subset \mathbb{N}$ such that $\lim_{K\ni k\rightarrow \infty} x^k = x^*$. Then, it follows from Proposition \ref{prop-descent} that $\lim_{K\ni k\rightarrow \infty} y^k = x^*$. Note that $x^k \in \mathcal{T}_{\gamma_k}(y^k)$ for some $\gamma_k \in [\frac{1}{L+\alpha},\overline{\gamma}]$. If necessary taking a subsequence, we can assume that $\lim_{K\ni k\rightarrow \infty} \gamma_k = \gamma\in [\frac{1}{L+\alpha},\overline{\gamma}]$. The desired result follows from \cite[Theorem 1.25]{RW09}. 
  
  \noindent
  (ii) By part (i), for each $x^*\in \omega(x^0)$, there exists $\gamma \in [\frac{1}{L+\alpha},\overline{\gamma}]$ such that $x^* \in \mathcal{T}_{\gamma}(x^*)$, which means that the function $\Gamma^{f,g}$ in \cite[Definition 3.2]{Themelis18} satisfies $\Gamma^{f,g}(x^*)\ge\frac{1}{L+\alpha}$. Together with \cite[Theorem 4.7]{Themelis18} and $\widetilde{L}^{-1} < \frac{1}{L+\alpha}$, there exists $\epsilon_{x^*} > 0$ such that $\mathbb{B}(x^*,\epsilon_{x^*})\subset\mathcal{O}$, on which $\mathcal{R}_{\widetilde{L}^{-1}}$ is single-valued and locally Lipschitz continuous.

  \noindent
  (iii) Let $\Gamma\!:=\bigcup_{x^*\in \omega(x^0)}\mathbb{B}(x^*,\epsilon_{x^*})$, where $\epsilon_{x^*}$ is the same as in (ii). We claim that $x^k\in\Gamma$ for $k>\overline{k}$ (if necessary by enlarging $\overline{k}$). If not, there will exist a subsequence $\{x^{k_j}\}_{j\in\mathbb{N}}$ such that ${\rm dist}(x^{k_j}, \Gamma) >1/j$ for all $j$. As $\{x^{k_j}\}_{j\in\mathbb{N}}$ is bounded by Assumption \ref{ass-bounded}, it must have a cluster point $\overline{x}^*\in \omega(x^0)$, so there exists $\overline{j}\in \mathbb{N}$ such that for $j>\overline{j}$, $x^{k_j} \in \mathbb{B}(\overline{x}^*,\epsilon_{\overline{x}^*})\subset\Gamma$, a contradiction to ${\rm dist}(x^{k_j}, \Gamma) >1/j$ for all $j$. Thus, there exists $\overline{k}\in\mathbb{N}$ such that $x^k\in\Gamma$ for all $k>\overline{k}$. By part (ii), $\mathcal{R}_{\widetilde{L}^{-1}}$ is single-valued and locally Lipschitz continuous on $\Gamma$, so is the mapping $\mathcal{T}_{\widetilde{L}^{-1}}$. This means that for all $k>\overline{k}$ step (2a) of Algorithm \ref{hybrid} holds with $m=0$. Consequently, $\widetilde{\gamma}_k = \widetilde{L}^{-1}$ for all $k>\overline{k}$.

  \noindent
  (iv) Pick any $x^*\in\omega(x^0)$. There exists $K \subset \mathbb{N}$ such that $\lim_{K\ni k\rightarrow \infty} x^k = x^*$. Along with Proposition \ref{prop-descent}, $\lim_{K\ni k\rightarrow \infty} x^{k+1}=x^*=\lim_{K\ni k\rightarrow \infty} y^{k+1}$. From Proposition \ref{prop-descent} and the lsc of $F$, we have $\overline{F}\ge F(x^*)\ge F_{\widetilde{L}^{-1}}(x^*)$, where the second inequality is by Lemma \ref{lemma-PGdescent}.  
  On the other hand, from Lemma \ref{lemma-bounded} (iii) and part (iii), the sequence $\{F_{\widetilde{L}^{-1}}(x^k)\}_{k>\overline{k}}$ is descent, and is also lower bounded because $F(x^{k+1})\leq F_{\widetilde{L}^{-1}}(x^k)$ for any $k>\overline{k}$ by \eqref{descent-K21} and part (iii). Hence, the sequence $\{F_{\widetilde{L}^{-1}}(x^k)\}_{k>\overline{k}}$ is convergent, and passing $K\ni k\to\infty$ to the inequality $F(x^{k+1})\leq F_{\widetilde{L}^{-1}}(x^k)$ leads to $\lim_{k\to\infty}F_{\widetilde{L}^{-1}}(x^k)=F_{\widetilde{L}^{-1}}(x^*)\ge\overline{F}$, where the equality follows by the continuity of $F_{\widetilde{L}^{-1}}(\cdot)$\cite[Theorem 1.25]{RW09}. The two sides show that the desired result holds. 
  \end{proof}
 
 Proposition \ref{prop-supp} (iii) implies that for $k>\overline{k}$, the semismooth Newton step of Algorithm \ref{hybrid} always solves $0\in\mathcal{R}_{\widetilde{L}^{-1}}(x)$. Next, we focus on the convergence of $\{x^k\}_{k\in\mathbb{N}}$ under the KL property of $F$ with exponent. To this end, write 
 \begin{equation}\label{Fstar}
   F_*\!:= F_{\widetilde{L}^{-1}},\, \mathcal{T}_*\!:=\mathcal{T}_{\widetilde{L}^{-1}},\, \mathcal{R}_*\!:=\mathcal{R}_{\widetilde{L}^{-1}}\ \ {\rm and}\ \mathcal{X}^*\!:=\big\{x\in\mathbb{R}^n\,|\, 0\in \mathcal{R}_{*}(x)\big\}.
 \end{equation}
 By Definition \ref{def-Lspoint} and Proposition \ref{prop-supp} (iii), $ \mathcal{X}^*=\!\big\{x\in\mathbb{R}^n\,|\, x\in \mathcal{T}_{*}(x)\big\}\ne\emptyset$.
\begin{theorem}\label{thm-gconvergence}
 Suppose that $F$ is a KL function of exponent $\theta\in [\frac{1}{2}, \frac{1}{2-\varrho}]$. Then under Assumptions \ref{ass-g}-\ref{ass-bounded}, it holds that $\sum_{k=\overline{k}+1}^{\infty}\|x^{k+1}-x^k\|<\infty$ and $\sum_{k=\overline{k}+1}^{\infty}\|y^{k+1}-y^k\|<\infty$, where $\overline{k}$ is the one in Proposition \ref{prop-supp} (iii), and consequently, $\{x^k\}_{k\in\mathbb{N}}$ and $\{y^k\}_{k\in\mathbb{N}}$ are convergent and converge to the same $x^*\in\mathcal{X}^*$.
 \end{theorem}
 \begin{proof}
 If there exists $\widetilde{k}\in\mathbb{N}$ such that $F(x^{\widetilde{k}}) = F(x^{\widetilde{k}+1})$, by Proposition \ref{prop-descent}, we have $x^{\widetilde{k}} = y^{\widetilde{k}}$, so $x^{\widetilde{k}}$ meets the stopping criterion. The conclusion follows. Next we assume that $F(x^k) \neq F(x^{k+1})$ for all $k\in\mathbb{N}$, and proceed the proof by two steps.

 \noindent
 {\bf Step 1}: To prove that $\sum_{k=1}^{\infty}\|x^k- y^k\|<\infty$. 
 From Algorithm \ref{hybrid}, for all $k\in\mathbb{N}$, $x^k\in\mathcal{T}_{\gamma_k}(y^k)=\mathcal{P}_{\gamma_k}g(y^k-\!\gamma_k\nabla\!f(y^k))$, which implies that $0\in \nabla\!f(y^k) + \gamma_k^{-1}(x^k- y^k) + \partial g(x^k)$ or equivalently $\nabla\!f(x^k)-\nabla\!f(y^k)-\gamma_k^{-1}(x^k- y^k) \in \partial F(x^k)$. Since $\gamma_k \geq \frac{1}{L+\alpha}$ by step (1a),  
 together with the Lipschitz continuity of $\nabla\!f$ on $\mathcal{O}$, we have for all $k>\overline{k}$, 
  \begin{equation}\label{subdiff-gap}
  {\rm dist}(0,\partial F(x^k))\leq  (L+\gamma_k^{-1})\|x^k - y^k\| \leq  (L+L+\alpha)\|x^k-y^k\| \leq 2\widetilde{L}\|x^k-y^k\|.
  \end{equation}
  Recall that $\omega(x^0)$ is nonempty and compact. Also, from Proposition \ref{prop-supp} (iv), $F(x)=\overline{F}$ for all $x\in\omega(x^0)$. As $F$ is a KL function, from \cite[Lemma 6]{Bolte14}, there exist $\varepsilon>0,\eta>0$ and $\phi\in\Upsilon_{\!\eta}$ such that for all $x\in[\overline{F}<F<\overline{F}+\eta]\cap\big\{x\in\mathbb{R}^n\ |\ {\rm dist}(x,\omega(x^0))\le\varepsilon\big\}$, 
 \[
  \phi'(F(x)-\overline{F}) {\rm dist}(0,\partial F(x)) \geq 1.
 \]
 By Proposition \ref{prop-descent}, $\{F(x^k)\}_{k\in\mathbb{N}}$ is a descent sequence converging to $\overline{F}$. Then, from the above equation, for all $k>\overline{k}$ (if necessary by enlarging $\overline{k}$), 
 \begin{equation}\label{KL-inequality1}
  \phi'(F(x^k)-\overline{F}) {\rm dist}(0,\partial F(x^k)) \geq 1.
 \end{equation}
 For each $k\in\mathbb{N}$, write $\Delta_{k}\!:= \phi(F(x^{k})-\overline{F})$. From the concavity of $\phi$, it holds that
 \begin{equation}\label{KL-inequality2}
  \Delta_{k}-\Delta_{k+1}\geq\phi'(F(x^k) \!-\!\overline{F})(F(x^k)\!-\! F(x^{k+1}))\quad {\rm for\ each}\ k\in\mathbb{N}.
 \end{equation}  
 Combining equations \eqref{subdiff-gap}-\eqref{KL-inequality2} with Proposition \ref{prop-descent} yields that for all $k>\overline{k}$, 
 \[
 \Delta_{k} - \Delta_{k+1} \geq \frac{F(x^k)- F(x^{k+1})}{{\rm dist}(0,\partial F(x^k))}  \geq \frac{\alpha}{4\widetilde{L}} \frac{\|x^{k+1}- y^{k+1}\|^2}{\|x^k-y^k\|}.
 \]
 Set $\widetilde{\alpha}:={4\widetilde{L}}/{\alpha}$. Using $2\sqrt{ab}\le a+b$ for $a=\widetilde{\alpha} (\Delta_k\!-\!\Delta_{k+1})$ and $b=\|x^k-y^k\|$ leads to
 \[
  \|x^{k+1}\!-\!y^{k+1}\|\le\sqrt{\widetilde{\alpha}(\Delta_k\!-\! \Delta_{k+1})\|x^k-y^k\|} \leq \frac{1}{2}\big[\widetilde{\alpha}(\Delta_k\!-\!\Delta_{k+1})+\|x^k\!-\!y^k\|\big]
 \]
 for all $k>\overline{k}$. 
 Summing this inequality from $k=\overline{k}+1$ to any $l>\overline{k}+1$ yields that
 \begin{equation*}
 \sum_{k=\overline{k}+1}^{l}\|x^{k+1}-y^{k+1}\| \leq \|x^{\overline{k}+1} - y^{\overline{k}+1}\| + \widetilde{\alpha}\sum_{k=\overline{k}+1}^{l}\!\!(\Delta_k-\Delta_{k+1})\leq \|x^{\overline{k}+1} - y^{\overline{k}+1}\| + \widetilde{\alpha}\Delta_{\overline{k}+1},
 \end{equation*} 
 where the last inequality is due to $\Delta_{k}\ge 0$. Passing the limit $l\to\infty$ yields the result.

 \noindent
 {\bf Step 2}: To prove that $\sum_{k=\overline{k}+1}^{\infty}\|\mathcal{R}_{*}(x^k)\|^{\varrho}<\infty$. By Proposition \ref{prop-supp} (iii) and \eqref{Fstar}, for all $k>\overline{k}$, $  F_{\widetilde{\gamma}_k} \equiv F_*$. From Proposition \ref{prop-supp} (iv) and its proof, $\{F_*(x^k)\}_{k>\overline{k}}$ is a descent sequence converging to $\overline{F}$. From the given assumption on $F$ and \cite[Remark 5.1 (ii)]{yu2022kurdyka}, we know that $F_*$ is a KL function of exponent $\theta\in[\frac{1}{2},\frac{1}{2-\varrho}]$. By following the arguments similar to those of Step 1, for $k>\overline{k}$ (if necessary by increasking $\overline{k}$),
 \begin{equation*}
 \phi'(F_*(x^k)-\overline{F}){\rm dist}(0,\partial F_*(x^k))\geq 1, \ {\rm with} \ \phi(t) = ct^{1-\theta} \ {\rm and} \ c>0. 
 \end{equation*}
 Raising both sides of the equation to the power of $2-\varrho$ yields
 \begin{align*}
  &(F_*(x^k)-\overline{F})^{-\theta(2-\varrho)}[{\rm dist}(0,\partial F_*(x^k))]^{2-\varrho}\ge [c(1-\theta)]^{\varrho-2}\\
  \Leftrightarrow\ &\widetilde{\phi}'(F_*(x^k)-\overline{F})[{\rm dist}(0,\partial F_*(x^k))]^{2-\varrho}\ge 1 \ {\rm with} \ \widetilde{\phi}(t)\!:=\frac{t^{1-\theta(2-\varrho)}}{[c(1-\theta)]^{\varrho-2}[1-\theta(2-\varrho)]}.
 \end{align*}
 For each $k>\overline{k}$, let $\widetilde{\Delta}_{k}\!:= \widetilde{\phi}(F_{*}(x^k) - \overline{F})$. Note that the function $\widetilde{\phi}$ is concave on $(0,\infty)$ due to $\theta\in [\frac{1}{2},\frac{1}{2-\varrho}]$. Then, for $k>\overline{k}$,  \begin{equation*}
  \widetilde{\Delta}_{k}-\widetilde{\Delta}_{k+1}
  \geq \widetilde{\phi}'(F_{*}(x^k)- \overline{F}) (F_{*}(x^k)- F_{*}(x^{k+1})) \geq \frac{F_*(x^{k}) - F_*(x^{k+1})}{[{\rm dist}(0,\partial F_{*}(x^k))]^{2-\varrho}}.
 \end{equation*}  
 By Proposition \ref{prop-supp} (iii), $\mathcal{R}_*$ is single-valued and locally Lipschitz continuous at every $x^k$ for $k\!>\!\overline{k}$. Along with Proposition \ref{Forder-Fgam}, for $k\!>\!\overline{k}$, $\partial F_{*}(x^k)\!=\! \{Q_{\widetilde{\gamma}_k}(x^k)\mathcal{R}_*(x^k)\}$. From the above inequality and Lemma \ref{lemma-bounded} (iii), for $k\!>\!\overline{k}$, 
  \begin{equation*}
  \widetilde{\Delta}_{k}-\widetilde{\Delta}_{k+1}\geq \frac{F_*(x^{k}) - F_*(x^{k+1})}{\|Q_{\widetilde{\gamma}_k}(x^k)\mathcal{R}_*(x^k)\|^{2-\varrho}} \geq\frac{\overline{c}_1\|\mathcal{R}_{*}(x^k)\|^2}{\|Q_{\widetilde{\gamma}_k}(x^k)\|_2^{2-\varrho} \|\mathcal{R}_{*}(x^k)\|^{2-\varrho}}=\frac{\overline{c}_1\|\mathcal{R}_{*}(x^k)\|^{\varrho}}{\|Q_{\widetilde{\gamma}_k}(x^k)\|_2^{2-\varrho}}.
  \end{equation*}
  Note that for each $k > \overline{k}$, $\|Q_{\widetilde{\gamma}_k}(x^k)\|_2\le 1\!+\!\widetilde{\gamma}_k\|\nabla^2\! f(x^k)\|_2\le 1\!+\!\widetilde{L}^{-1}L\le 2$. Then, 
  \begin{equation}\label{temp-ineqMR}
  \sum_{k=\overline{k}+1}^{\infty}\|\mathcal{R}_{*}(x^k)\|^{\varrho} \leq 2^{2-\varrho}\,\overline{c}_1^{-1}  \sum_{k=\overline{k}+1}^{\infty}(\widetilde{\Delta}_{k}-\widetilde{\Delta}_{k+1}) \leq 2^{2-\varrho}\,\overline{c}_1^{-1}\widetilde{\Delta}_{\overline{k}+1}<\infty,
 \end{equation}
 where the second inequality holds by $\widetilde{\Delta}_k\ge 0$ for $k>\overline{k}$. The claimed fact holds.

 Now from the definitions of $y^{k+1}$ and $\|d^k\|$ in step (2c), for any $l>\overline{k}+1$, we have
 \begin{equation*}
 \sum_{k=\overline{k}+1}^{l}\|x^k\!-\!y^{k+1}\|\leq \sum_{k=\overline{k}+1}^{l}\big(\|d^k\|+\|x^k-\mathcal{T}_*(x^k)\|\big)\leq \!\! \sum_{k=\overline{k}+1}^{l}\!\!\big(\overline{\varsigma}\|\mathcal{R}_*(x^k)\|^{\varrho}+\widetilde{L}\|\mathcal{R}_*(x^k)\|\big).
 \end{equation*}
 Passing the limit $l\to\infty$ and using  \eqref{temp-ineqMR} leads to $\sum_{k=\overline{k}+1}^{\infty}\|x^k-\!y^{k+1}\|<\infty$. 
 Note that $\|x^k-x^{k+1}\| \leq \|x^{k+1}-y^{k+1}\|+\|x^{k} - y^{k+1}\|$ and $\lim_{k\rightarrow \infty}\|x^k-y^k\|= 0$ (Proposition \ref{prop-descent}), which combines the results in Steps 1-2 yields that $ \sum_{k=\overline{k}+1}^{\infty}\|x^k-\!x^{k+1}\|<\infty$
 and $\sum_{k=\overline{k}+1}^{\infty}\|y^k-\!y^{k+1}\|<\infty$.  Therefore, both sequences are convergent and share the same limit, which belongs to $\mathcal{X}^*$ by Proposition \ref{prop-supp} (i). The proof is completed.
 \end{proof}
 \begin{remark}
 (i) It is worth pointing out that Themelis et al.  \cite{Themelis18} achieved the convergence of the iterate sequence generated by a hybrid of PG and quasi-Newton methods by requiring that the direction $d^k$ is upper bounded by the residual, apart from the KL property of the objective function. In \cite{wu23, wu23b}, the authors achieved the full convergence of the iterate sequences generated by a hybrid of PG and Newton methods by imposing a curvature ratio assumption on the direction, as well as the KL property of the objective function. The assumptions on the direction in these works are not verifiable when the generalized Hessian is not uniformly positive definite. In Theorem \ref{thm-gconvergence}, we prove the convergence of the iterate sequence generated by Algorithm \ref{hybrid} under the KL property of the objective function with exponent $\theta\in[\frac{1}{2},\frac{1}{2-\varrho}]$. 

 \noindent
 (ii) By Definition \ref{KL-Def}, if $F$ is a KL function of exponent $\theta<1$, then it is also a KL function of exponent $[\theta, 1)$. Obviously, if $F$ is a KL function of exponent $1/2$, then it is necessarily a KL function of exponent $\theta\in[\frac{1}{2},\frac{1}{2-\varrho}]$. In \cite{wu21}, some zero-norm regularized composite functions are shown to be a KL function of exponent $1/2$. 
 \end{remark}
 \section{Local superlinear convergence rate}\label{sec5}

 To proceed with the local superlinear convergence rate analysis of Algorithm \ref{hybrid}, we make the following assumption.
 \begin{assumption}\label{ass-local}
  {\rm (i)} $\{x^k\}_{k\in\mathbb{N}}$ converges to $x^*\in\mathcal{X}^*$ that is a local minimum of $F$.

  \noindent
 {\rm (ii)} The mapping $\mathcal{R}_{*}$ is metrically subregular at $x^*$ for the origin, i.e., there exist $\varepsilon>0$ and $\kappa>0$ such that for all $x\in\mathbb{B}(x^*,\varepsilon)\cap \mathcal{O}$, $ {\rm dist}(x,\mathcal{X}^*) \leq \kappa\,{\rm dist}(0, \mathcal{R}_{*}(x))$.

 \noindent
 {\rm (iii)} There exists $\overline{\varepsilon}\in(0,\varepsilon)$ such that for all $x\in \mathbb{B}(x^*,\overline{\varepsilon})\cap \mathcal{X}^*$, $F_*(x) = F_*(x^*)$.
 \end{assumption}

 Assumption \ref{ass-local} (i) seems restrictive because $x^*\in\mathcal{X}^*$ is required to be a local minimum of $F$, but it does not require the isolatedness of $x^*$. Such an assumption has been used for the local convergence analysis of the hybrid algorithms in \cite{Themelis18,themelis21}. As $x^*$ is a local minimum of $F_{*}$ by \cite[Theorem 4.4]{themelis21}, it follows from \cite[Proposition 2 (ii)]{liu2024inexact} that the KL property of $F_{*}$ with exponent $1/2$ at $x^*$ implies the metric subregularity of $\partial F_{*}$ at $x^*$ for the origin, and thus that of $\mathcal{R}_*$ at $x^*$ for the origin by Lemma \ref{lemma-subregularity-growth} (i). Thus, together with \cite[Remark 5.1 (ii)]{yu2022kurdyka}, we conclude that the KL property of $F$ with exponent $1/2$ at $x^*$ implies Assumption \ref{ass-local} (ii). Assumption \ref{ass-local} (iii) is weaker than the similar ones concerning separation of stationary values used in \cite{Luo93,lipong18} by recalling that $\mathcal{X}^*$ is the set of $L$-stationary point corresponding to stepsize $\gamma = \widetilde{L}^{-1}$, a common condition in the local convergence analysis of algorithms for nonconvex optimization. The above discussions show that Assumption \ref{ass-local} is rather mild. 
  
 By Proposition \ref{prop-supp} (ii), under Assumptions \ref{ass-g}-\ref{ass-local}, $\mathcal{R}_{*}$ is locally Lipschitz continuous at $x^*$, and we denote by  $L_{\mathcal{R}_*}$ its Lipschitz modulus at $x^*$. Next we show that under Assumptions \ref{ass-g}-\ref{ass-local}, if $\{d^k\}_{k\in\mathbb{N}}$ is a sequence of superlinear directions\footnote{The $\{d^k\}_{k\in\mathbb{N}}$ is called a sequence of superlinear directions if
 $\lim_{k\rightarrow\infty} \frac{{\rm dist}(x^k +d^k, \mathcal{X}^*)}{{\rm dist}(x^k, \mathcal{X}^*)}=0$.} (see \cite[Definition 5.9]{themelis21}), the semismooth Newton step finally takes the unit step-size.
 \begin{lemma}\label{lemma-unit-step}
  Under Assumptions \ref{ass-g}-\ref{ass-local}, if $\{d^k\}_{k\in\mathbb{N}}$ is a sequence of superlinear directions, then the semismooth Newton step finally takes the unit step-size, i.e., $l_k= 0$ in step (2c) of Algorithm \ref{hybrid}.
 \end{lemma}
 \begin{proof}
  It is clear that $\omega(x^0) \subset \mathcal{O}$, which together with Proposition \ref{prop-dkzero} (i) implies that if necessary enlarging $\overline{k}$, $x^k+d^k \in \mathcal{O}$ for $k>\overline{k}$, and thus $F_*(x^k+d^k)$ is well defined.
  By Lemma \ref{lemma-subregularity-growth} (ii) and Assumption \ref{ass-local} (ii), $F_*$ possesses a quadratic growth at $x^*$, so there exists $\kappa_0>0$ such that 
 \begin{equation}\label{varphilam_quadgrowth}
  F_{*}(x^k) - F_{*}(x^*) \geq \kappa_0{\rm dist}(x^k, \mathcal{X}^*)^2, \ {\rm for \ } k\geq \overline{k} \ {\rm (if \ necessary \ by \ enlarging \ } \overline{k}).
 \end{equation}
 For each $k>\overline{k}$, taking $x^{k,*}\in {\rm proj}_{\mathcal{X}^*}(x^k\!+\!d^k)$, we have $\|x^{k,*}- x^*\|\leq \|x^k+d^k-x^{k,*}\| + \|x^k+d^k - x^*\| \leq 2\|x^k+d^k - x^*\|$. Using Proposition \ref{prop-dkzero} (i) and $\lim_{k\to\infty}x^k=x^*$ leads to $\lim_{k\rightarrow \infty} \|x^{k,*}- x^*\| =0$. By Assumption \ref{ass-local} (iii), $F(x^{k,*}) = F_{*}(x^{k,*}) = F_{*}(x^*)$ for $k>\overline{k}$ (if necessary by enlarging $\overline{k}$). For each $k>\overline{k}$, from $x^{k,*}\in {\rm proj}_{\mathcal{X}^*}(x^k\!+\!d^k)$, we have $x^{k,*}\in\mathcal{T}_{\widetilde{L}^{-1}}(x^{k,*})$. Invoking \eqref{def-fbe} with $x=x^k + d^k$ and $\gamma=\widetilde{L}^{-1}$  yields
 \begin{align}\label{FBE-descent-lemma}
  F_{*}(x^k + d^k)-F(x^{k,*})
  &\le f(x^k+d^k)+\langle\nabla\!f(x^k+d^k), x^{k,*}-(x^k+d^k)\rangle\nonumber\\
  &\quad+({\widetilde{L}}/{2})\|x^{k,*} - (x^k+d^k)\|^2-f(x^{k,*})\nonumber\\
 & \leq (\widetilde{L}/2+L/2)\|x^{k,*} - (x^k+d^k)\|^2 \leq \widetilde{L}\|x^{k,*} - (x^k+d^k)\|^2,
 \end{align}
 where the second inequality is obtained by using Assumption \ref{ass-g} (i) and the descent lemma. From the above \eqref{varphilam_quadgrowth} and \eqref{FBE-descent-lemma}, it follows that for any $k>\overline{k}$, 
 \begin{equation}\label{eq-superlinear}
 \varepsilon_k := \frac{F_{*}(x^k+d^k) - F_{*}(x^*)}{F_{*}(x^k) - F_{*}(x^*)} \leq \frac{\widetilde{L}~\!{\rm dist}(x^k +d^k, \mathcal{X}^*)^2}{\kappa_0{\rm dist}(x^k, \mathcal{X}^*)^{2}}
 \end{equation}
 which, by using the condition that $\{d^k\}_{k>\overline{k}}$ is a sequence of superlinear directions, means that $\varepsilon_k\to 0$ as $k\to\infty$. Without loss of generality, we assume that $\varepsilon_k\in(0,\frac{1}{2})$ for $k>\overline{k}$. Next we claim that for sufficiently large $k$, $F(\mathcal{T}_{*}(x^k)) \geq F_*(x^*)$. From Proposition \ref{prop-dkzero} (i) and Lemma \ref{lemma-bounded} (ii), it follows that $\lim_{k\to\infty}\|\mathcal{T}_{*}(x^k)-x^k\|=0$ and then $\lim_{k\to\infty}\|\mathcal{T}_{*}(x^k)-x^*\|=0$. Since $x^*$ is a local minimum of $F$, we have $F(\mathcal{T}_{*}(x^k)) \geq F(x^*) = F_*(x^*)$ for all $k>\overline{k}$ (if necessary by enlarging $\overline{k}$), where the equality is by Proposition \ref{prop-supp} (iv). The claimed inequality holds.
 Now from \eqref{eq-superlinear}, 
 \begin{align*}
  F_{*}(x^k+d^k) - F_{*}(x^k) & = F_{*}(x^k+d^k) - F_{*}(x^*) - (F_{*}(x^k)-F_{*}(x^*)) \\
  & = (\varepsilon_k-1) (F_{*}(x^k)-F_{*}(x^*))
   \leq (\varepsilon_k-1) (F_{*}(x^k)-F (\mathcal{T}_{*}(x^k))) \\
  & \leq (\varepsilon_k - 1) \frac{\widetilde{\gamma}_k-\widetilde{\gamma}_k^2 L}{2} \|\mathcal{R}_{\widetilde{\gamma}_k}(x^k)\|^2 < -\frac{\widetilde{\gamma}_k-\widetilde{\gamma}_k^2 L}{4} \|\mathcal{R}_{\widetilde{\gamma}_k}(x^k)\|^2,
 \end{align*}
 where the second inequality follows Lemma \ref{lemma-PGdescent} (ii). This shows that for $k>\overline{k}$, step (2c) holds with $l= 0$, i.e., the semismooth Newton step takes the unit step-size. 
 \end{proof}

Next under Assumptions \ref{ass-g}-\ref{ass-local} we provide some conditions to ensure that $\{d^k\}_{k>\overline{k}}$ is a sequence of superlinear directions, where $\overline{k}$ is the one in Proposition \ref{prop-supp}. These conditions are stated as follows:
 \begin{enumerate}
 \item[{\it (C1)}] there exist $\varepsilon_1>0,\kappa_1>0$ and $\rho_1\in (\tau,1]$ such that for any $x\in\mathbb{B}(x^*,\varepsilon_1)\cap \mathcal{O}$ and $\overline{x}\in {\rm proj}_{\mathcal{X}^*}(x)$, the inclusion $\partial^2 F_{*}(x) \subset \partial^2 F_{*}(\overline{x}) +\kappa_1\|x-\overline{x}\|^{\rho_1}\mathbb{B}$ holds;
         
 \item[{\it (C2)}] there exists $\delta>0$ such that $H\succeq 0$ for every $H\in\bigcup_{x\in \mathcal{X}^*\cap \mathbb{B}(x^*,\delta)}\partial^2 F_{*}(x)$; 

 \item[{\it (C3)}] there exist $\varepsilon_2>0,\kappa_2>0$ and $\rho_2\in [\tau,1]$ such that for every $x,y\in \mathbb{B}(x^*,\varepsilon_2)\cap \mathcal{O}$ and $J_x \in \partial_C \mathcal{R}_{*}(x)$, $\|\mathcal{R}_*(y) -\mathcal{R}_*(x) - J_x(y-x)\| \leq \kappa_2\|y-x\|^{1+\rho_2}.$      
 \end{enumerate}

 \medskip
 We will give some remarks toward these three conditions at the end of this section. Before proving that $\{d^k\}_{k>\overline{k}}$ is a sequence of superlinear directions, we establish a crucial property of $d^k$ for $k\!>\!\overline{k}$ under Assumptions \ref{ass-g}-\ref{ass-local} and conditions {\it (C1)-(C2)}. 
 \begin{proposition}\label{prop-dk}
  For each $k\in\mathbb{N}$, let $G_k := H_k+\sigma_k\mu_kI$. Suppose that Assumptions \ref{ass-g}-\ref{ass-local} and {\it (C1)-(C2)} hold, Then for each $k>\overline{k}$, the optimal solution $d^k$ in step (2b) of Algorithm \ref{hybrid} satisfies $G_kd^k+Q_{\widetilde{\gamma}_k}(x^k)\mathcal{R}_{\widetilde{\gamma}_k}(x^k)=0$. 
 \end{proposition}
 \begin{proof}
  We claim that $[-\lambda_{\min}(H_k)]_+\leq\kappa_1[{\rm dist}(x^k,\mathcal{X}^*)]^{\rho_1}$ for $k>\overline{k}$ (if necessary by increasing $\overline{k}$). For each $k>\overline{k}$, pick any $\overline{x}^k\in{\rm proj}_{\mathcal{X}^*}(x^k)$. Then, it holds that
 \begin{equation}\label{xbark-limit}
 \lim_{k\rightarrow \infty}\|\overline{x}^k - x^*\| = 0
 \end{equation}
 because $\|\overline{x}^k - x^*\|\leq \|\overline{x}^k-x^k\| + \|x^k - x^*\| \leq 2 \|x^k- x^*\|$ and $\lim_{k\to\infty}x^k=x^*$. From condition {\it (C1)}, for every $k>\overline{k}$ (if necessary by enlarging $\overline{k}$) and $H_k\in\partial^2 F_{*}(x^k)$, there exists $\overline{H}_k \in \partial^2 F_{*}(\overline{x}^k)$ such that $\|H_k - \overline{H}_k\|_2\le\kappa_1\|x^k-\overline{x}^k\|^{\rho_1} = \kappa_1{\rm dist}(x^k,\mathcal{X}^*)^{\rho_1}$. Together with condition {\it (C2)}, $\overline{H}_k \succeq 0$.  Fix any $k>\overline{k}$. If $\lambda_{\min}(H_k) \geq 0$, the desired claim is trivial, so it suffices to consider that $\lambda_{\min}(H_k) < 0$. From Weyl's inequality (see \cite[Corollary III.2.6]{bhatia2013matrix}), it follows that for every $k>\overline{k}$, 
  \begin{align*}
  [-\lambda_{\min}(H_k)]_+\! =\! -\lambda_{\min}(H_k)\! \leq \!\lambda_{\min}(\overline{H}_k) \!-\!\lambda_{\min}(H_k) \!\leq\! \|H_k \!-\! \overline{H}_k\|_2\!\leq\! \kappa_1{\rm dist}(x^k,\mathcal{X}^*)^{\rho_1},
  \end{align*}
  where the first inequality is due to $\lambda_{\min}(\overline{H}_k)\ge 0$. The claimed conclusion holds.
  
 From the expression of $G_k$, for every $k>\overline{k}$ and $d\in\mathbb{R}^n$, it holds that
 \begin{align}\label{temp-ineqGk}
  d^{\top}G_kd&= d^{\top}(H_k +\mu_k\sigma_k I)d\ge \lambda_{\rm min}(H_k)\|d\|^2\!+\!\sigma_k\|\mathcal{R}_{\widetilde{\gamma}_k}(x^k)\|^{\tau}\|d\|^2,
 \end{align}
 where the inequality is by the definition of $\mu_k$. From the above claimed conclusion and Assumption \ref{ass-local} (ii), for $k>\overline{k}$, $[-\lambda_{\min}(H_k)]_+ \le\kappa_1[{\rm dist}(x^k,\mathcal{X}^*)]^{\rho_1} \leq \kappa_1\kappa^{\rho_1}\|\mathcal{R}_{\widetilde{\gamma}_k}(x^k)\|^{\rho_1}$, which implies that 
  $\lambda_{\min}(H_k) \geq -\kappa_1\kappa^{\rho_1} \|\mathcal{R}_{\widetilde{\gamma}_k}(x^k)\|^{\rho_1}$. Along with \eqref{temp-ineqGk}, for every $k>\overline{k}$ (enlarging $\overline{k}$ if necessary) and $d\in\mathbb{R}^n$, 
 \begin{equation}\label{Gkpd}
  d^{\top}G_kd \geq \left(\sigma_k\|\mathcal{R}_{\widetilde{\gamma}_k}(x^k)\|^{\tau} -\kappa_1\kappa^{\rho_1}\|\mathcal{R}_{\widetilde{\gamma}_k}(x^k)\|^{\rho_1}\right)\|d\|^2 \geq \frac{1}{2}\sigma_k\|\mathcal{R}_{\widetilde{\gamma}_k}(x^k)\|^{\tau}\|d\|^2,
  \end{equation}
  where the second inequality uses $\rho_1\in (\tau,1]$ and $\lim_{ k\rightarrow \infty}\|\mathcal{R}_{\widetilde{\gamma}_k}(x^k)\|= 0$ by Proposition \ref{prop-dkzero} (i). 
  Then, $G_k\succeq \frac{1}{2}\sigma_k\|\mathcal{R}_{\widetilde{\gamma}_k}(x^k)\|^{\tau}I$. Let $\overline{d}^k = G_k^{-1}Q_{\widetilde{\gamma}_k}(x^k)\mathcal{R}_{\widetilde{\gamma}_k}(x^k)$. We have
  \begin{align*}
  \|\overline{d}^k\| & \leq 2\sigma_k^{-1}\|\mathcal{R}_{\widetilde{\gamma}_k}(x^k)\|^{-\tau}\|Q_{\widetilde{\gamma}_k}(x^k)\|_2 \|\mathcal{R}_{\widetilde{\gamma}_k}(x^k)\| \\
  & \leq  4\underline{\sigma}^{-1}\|\mathcal{R}_{\widetilde{\gamma}_k}(x^k)\|^{1-\tau}< \underline{\varsigma}\|\mathcal{R}_{\widetilde{\gamma}_k}(x^k)\|^{\varrho}\leq \varsigma_k\|\mathcal{R}_{\widetilde{\gamma}_k}(x^k)\|^{\varrho},
 \end{align*}
 where the second inequality is using  $\|Q_{\widetilde{\gamma}_k}(x^k)\|_2\leq 1+\widetilde{\gamma}_k\|\nabla^2\!f(x^k)\|_2 \leq 2$, and the third one is due to $\varrho\in(0,1-\tau)$ and $\lim_{k\rightarrow \infty}\|\mathcal{R}_{\widetilde{\gamma}_k}(x^k)\|= 0$ by Proposition \ref{prop-dkzero} (i). 
 By using the above strict inequality, it is not hard to check that $\overline{d}^k$ is an optimal solution of the problem in step (2b). From the above \eqref{Gkpd}, we know that the objective function of the problem in step (2b) is strongly convex on its feasible set, which implies that it has a unique optimal solution. This means that $d^k=\overline{d}^k$, which implies that $G_kd^k+Q_{\widetilde{\gamma}_k}(x^k)\mathcal{R}_{\widetilde{\gamma}_k}(x^k)=0$. 
 \end{proof}

 Now we are ready to prove that $\{d^k\}_{k\in\mathbb{N}}$ is a sequence of superlinear directions, and apply this result to achieve the local superlinear convergence rate of $\{y^k\}_{k\in\mathbb{N}}$. 
 \begin{theorem}\label{dir-slinear}
 Suppose that Assumptions \ref{ass-g}-\ref{ass-local} and conditions {\it (C1)-(C3)} hold. Then there exists $\widetilde{c}>0$ such that ${\rm dist}(x^k\!+d^k,\mathcal{X}^*)\leq \widetilde{c}\,{\rm dist}(x^k,\mathcal{X}^*)^{1+\tau}$ for all $k>\overline{k}$ (if necessary by increasing $\overline{k}$), so $\{d^k\}_{k>\overline{k}}$ is a sequence of superlinear directions and ${\rm dist}(y^{k},\mathcal{X}^*)$ converges to $0$  at a superlinear convergence rate.
 \end{theorem}
 \begin{proof}
  We prove the first part of the conclusions by the following two steps. 
 
 {\bf Step 1:} To prove the existence of $\widetilde{c}_1>0$ such that $\|d^k\| \leq \widetilde{c}_1{\rm dist}(x^k,\mathcal{X}^*)$ for $k>\overline{k}$. For each $k>\overline{k}$, let $\overline{x}^k \in {\rm proj}_{\mathcal{X}^*}(x^k)$, $v^k\!:= -Q_{\widetilde{\gamma}_k}(x^k)\mathcal{R}_*(x^k) - G_k(\overline{x}^k-x^k)$ and $w^k\!:= d^k - (\overline{x}^k-x^k)$. By Proposition \ref{prop-dk}, for $k>\overline{k}$, $v^k = G_k w^k$, which along with $G_k=H_k +\sigma_k\mu_k I$ and $H_k = Q_{\widetilde{\gamma}_k}(x^k) J_k$ for some $J_k \in \partial_C\mathcal{R}_{*}(x^k)$ implies that
\begin{equation}
    \begin{aligned}
        \|v^k\| & = \| Q_{\widetilde{\gamma}_k}(x^k) \mathcal{R}_*(x^k) + H_k (\overline{x}^k - x^k) + \mu_k\sigma_k (\overline{x}^k - x^k)\|\\
 & \leq \|Q_{\widetilde{\gamma}_k}(x^k)\|_2\|\mathcal{R}_*(x^k) + J_k (\overline{x}^k - x^k)\| + \mu_k\sigma_k \|\overline{x}^k - x^k\| \\
  & \leq 2\kappa_2\|\overline{x}^k- x^k\|^{1+\rho_2} + \overline{\sigma}\|\mathcal{R}_{*}(x^k)\|^{\tau}\|\overline{x}^k- x^k\|\leq (2\kappa_2+\overline{\sigma}L_{\mathcal{R}_*}^{\tau}){\rm dist}(x^k,\mathcal{X}^*)^{1+\tau},
  \end{aligned}
\end{equation}
 where the second inequality follows from {\it (C3)}, $\|Q_{\widetilde{\gamma}_k}(x^k)\|_2\leq 2$ and the definition of $\mu_k$, and the last one is due to $\rho_2\in[\tau,1]$ and $\|\mathcal{R}_*(x^k)\|=\|\mathcal{R}_*(x^k)-\mathcal{R}_*(\overline{x}^k)\|\leq L_{\mathcal{R}_*}\|x^k - \overline{x}^k\|$ by recalling that $L_{\mathcal{R}_*}$ is the Lipschitz constant of $\mathcal{R}_*$ around $x^*$.
 Fix any $k>\overline{k}$. From the proof of Proposition \ref{prop-dk}, $G_k \succeq \frac{1}{2}\sigma_k\|\mathcal{R}_{\widetilde{\gamma}_k}(x^k)\|^{\tau}$, for which 
 \[
 \|w^k\| \leq   2\sigma_k^{-1}\|\mathcal{R}_*(x^k)\|^{-\tau}\|v^k\| \leq 2\underline{\sigma}^{-1}\kappa^{-\tau}(2\kappa_2+\overline{\sigma}L_{\mathcal{R}_*}^{\tau}){\rm dist}(x^k,\mathcal{X}^*).
 \]
 Thus, $\|d^k\| \leq \|w^k\| + \|\overline{x}^k- x^k\| \leq [2\underline{\sigma}^{-1}\kappa^{-\tau}(2\kappa_2+\overline{\sigma}L_{\mathcal{R}_*}^{\tau})+1]{\rm dist}(x^k,\mathcal{X}^*)$. Consequently, the desired result holds with $\widetilde{c}_1:= 2\underline{\sigma}^{-1}\kappa^{-\tau}(2\kappa_2+\overline{\sigma}L_{\mathcal{R}_*}^{\tau})+1$.

{\bf Step 2:} To prove the existence of $\widetilde{c}>0$ such that for all $k>\overline{k}$, ${\rm dist}(x^k\!+d^k,\mathcal{X}^*)\leq \widetilde{c}\,{\rm dist}(x^k,\mathcal{X}^*)^{1+\tau}$. As $\lim_{k\to\infty}(x^k\!+\!d^k)=x^*$, for $k>\overline{k}$ (if necessary by increasing $\overline{k}$), 
\begin{align}\label{dist-MXstar}
 {\rm dist}(x^k\!+\!d^k, \mathcal{X}^*) & \leq \kappa \|\mathcal{R}_*(x^k\!+\!d^k)\|\leq  \kappa[L/(2\alpha)\!+1]\|Q_{\widetilde{\gamma}_k}(x^k)\mathcal{R}_*(x^k\!+\! d^k)\|,
 \end{align}
 where the first inequality uses Assumption \ref{ass-local} (ii), and the second one follows by \eqref{minieigen-Qk}.
 From Proposition \ref{prop-dk} and the expression of $G^k$, it follows that 
 $Q_{\widetilde{\gamma}_k}(x^k) \mathcal{R}_*(x^k)+H_kd^k + \mu_k\sigma_kd^k = 0$ for $k>\overline{k}$. Recall that $H_k\in\partial^2 F_{*}(x^k)$ for $k>\overline{k}$ and that $L_{\mathcal{R}_*}$ is the Lipschitz constant of $\mathcal{R}_*$ around $x^*$. By Proposition \ref{prop-CJac} there exists $J_k \in \partial_C \mathcal{R}_*(x^k)$ such that $H_k = Q_{\widetilde{\gamma}_k}(x^k)J_k$, then
 \begin{align*}
 & \quad \ \|Q_{\widetilde{\gamma}_k}(x^k)\mathcal{R}_*(x^k\!+\!d^k)\| \\
 &\leq \|Q_{\widetilde{\gamma}_k}(x^k)\mathcal{R}_*(x^k\!+\!d^k) - Q_{\widetilde{\gamma}_k}(x^k)\mathcal{R}_*(x^k) - H_kd^k\| + \mu_k\sigma_k \|d^k\|\\
 &\le \|Q_{\widetilde{\gamma}_k}(x^k)\|_2 \|\mathcal{R}_{*}(x^k+d^k) - \mathcal{R}_*(x^k) - J_kd^k\| + \mu_k\sigma_k\|d^k\|\\
 &\stackrel{{\rm (C3)}}{\le} 2\kappa_2\|d^k\|^{1+\rho_2} + \overline{\sigma}L_{\mathcal{R}_*}^{\tau}\|x^k - \overline{x}^k\|^{\tau}\|d^k\|\\
 &\stackrel{{\rm Step 1}}{\le} 2\kappa_2\widetilde{c}_1^{1+\rho_2}[{\rm dist}(x^k,\mathcal{X}^*)]^{1+\rho_2}+\overline{\sigma}L_{\mathcal{R}_*}^{\tau}\widetilde{c}_1[{\rm dist}(x^k,\mathcal{X}^*)]^{1+\tau}\\
 &\le (2\kappa_2\widetilde{c}_1^{1+\rho_2}+\overline{\sigma}L_{\mathcal{R}_*}^{\tau}\widetilde{c}_1)[{\rm dist}(x^k,\mathcal{X}^*)]^{1+\tau}
 \end{align*}
 where the third inequality is also using $\|Q_{\widetilde{\gamma}_k}(x^k)\|_2\le 2$ and the expression of $\mu_k$, and the last one is due to $\rho_2\in[\tau,1]$. Together with the above \eqref{dist-MXstar}, for $k>\overline{k}$, we have 
 ${\rm dist}(x^k\!+\!d^k, \mathcal{X}^*)\le \widetilde{c}\,{\rm dist}(x^k,\mathcal{X}^*)^{1+\tau}$
 with $\widetilde{c}:=\frac{\kappa(L+2\alpha)}{2\alpha}(2\kappa_2\widetilde{c}_1^{1+\rho_2}+\overline{\sigma}L_{\mathcal{R}_*}^{\tau}\widetilde{c}_1)$. 

 Now we have proved that $\{d^k\}_{k\in\mathbb{N}}$ is a sequence of superlinear directions. The rest focuses on the superlinear convergence rate of the sequence $\{y^k\}_{k\in\mathbb{N}}$. From Lemma \ref{lemma-unit-step} and the definition of $y^{k+1}$, we have $y^{k+1}=x^k+d^k$ for $k>\overline{k}$. Therefore, 
 \begin{equation}\label{superlineardk}
  \limsup_{k\to\infty} \frac{{\rm dist}(y^{k+1}, \mathcal{X}^*)}{{\rm dist}(x^k, \mathcal{X}^*)^{1+\tau}} \leq \widetilde{c}.
 \end{equation}
 Fix any $k>\overline{k}$. Let $\overline{x}^k \in {\rm proj}_{\mathcal{X}^*}(x^k)$ and $\overline{y}^k \in{\rm proj}_{\mathcal{X}^*}(y^k)$. Then, from \eqref{varphilam_quadgrowth}, we get
 \begin{align*}
  \kappa_2 \|x^k - \overline{x}^k\|^2 & \leq F_*(x^k) - F_*(x^*) \leq F(x^k) - F_*(x^*)  \leq F_{\gamma_k}(y^k) - F_*(x^*) \\
  & \stackrel{\eqref{FBEcomp}}{\le} F_{*} (y^k) - F_*(x^*)\stackrel{\eqref{FBE-descent-lemma}}{\le}  \widetilde{L}\|y^k - \overline{y}^k\|^2,
  \end{align*}
  where the third inequality is due to $x^k\in\mathcal{T}_{\gamma_k}(y^k)$ and Lemma \ref{lemma-PGdescent} (ii). Rearranging the above inequality leads to ${\rm dist}(x^k,\mathcal{X}^*) \leq \sqrt{\widetilde{L}\kappa_2^{-1}}  {\rm dist}(y^k,\mathcal{X}^*)$.
   Combining the above inequality with \eqref{superlineardk} results in 
   $\limsup_{k\to\infty} \frac{{\rm dist}(y^{k+1}, \mathcal{X}^*)}{{\rm dist}(y^k, \mathcal{X}^*)^{1+\tau}} \leq \widetilde{c}_2\ \ {\rm with}\ \  \widetilde{c}_2 :=\widetilde{c} (\widetilde{L}\kappa_2^{-1})^{\frac{1+\tau}{2}}.$
  Then, $\{{\rm dist}(y^{k},\mathcal{X}^*)\}_{k\in\mathbb{N}}$ converges to $0$ at a superlinear convergence rate.
 \end{proof}
 \begin{remark}
     We take a closer look at conditions {\it (C1)-(C3)} used for the local superlinear convergence of Algorithm \ref{hybrid}. 
 First, when $\rho_1 = 1$, condition {\it (C1)} is the calmness of multifunction $\partial^2 F_{*}$ at $\overline{x}$ \cite[Eq. 9(30)]{RW09}. When $f$ is a quadratic function and $g$ is a piece-wise linear quadratic function, $\partial^2F_*$ is a polyhedral multifunction by Proposition \ref{prop-CJac}, which implies that $\partial^2F_*$ is locally upper Lipschitzian at every $x\in\mathcal{O}$ by \cite[Proposition 1]{robinson1981some}.
 Condition {\it (C3)} is equivalent to the (b)-regularity of $(\mathcal{R}_*,\partial_{C}\mathcal{R}_{*})$ along $\mathcal{O}$ at $x^*$ with exponent $1+\rho_2$ (see \cite[Definition 1]{charisopoulos2024superlinearly}), which is actually the $\rho_2$-order uniform semismoothness of $\mathcal{R}_*$ at $x^*$. As will be shown in Proposition \ref{exam-proxmap} later, for the function $F$ with $f$ satisfying Assumption \ref{ass-g} and $g$ coming from \eqref{exam-g}, condition {\it(C3)} holds with $\rho_2=1$, and so does condition {\it (C1)} for $\rho_1=1$.

  Recall that when $f$ is a quadratic function, $\partial^2 F_*$ is identical to $\partial_C(\nabla F_{*})$.
  It was shown in \cite[Theorem 3.1]{hiriart1984generalized} that for $\mathcal{C}^{1,1}$ function $h$, if $\overline{x}$ is a local minimum of $h$, then there must exist $H \in \partial_C(\nabla h)(\overline{x})$ such that $H\succeq 0$. Based on this, Condition {\it(C2)} further assumes that $H\succeq 0$ for every $H\in \bigcup_{x\in\mathcal{X}^*\cap \mathbb{B}(x^*,\delta)}\partial^2 F_*(x)$. Suppose that $f$ is a quadratic function. 
  From \cite[Theorem 4.4]{themelis21}, if there exists $\varepsilon_1>0$ such that all $\overline{x}^*\in\mathbb{B}(x^*,\varepsilon_1)\cap\mathcal{X}^*$ are the local minimum of $F$, then they are also the local minimum of $F_{*}$; and if in addition $\partial^2 F_{*}$ is single-valued and continuously differentiable on a neighborhood of these $\overline{x}^*$, condition {\it(C2)} holds because $\partial^2 F_{*}(\overline{x}^*)=\{\nabla^2 F_{*}(\overline{x}^*)\}$ and $\nabla^2 F_{*}(\overline{x}^*)$ is positive semidefinite.
 \end{remark}

Now we provide two examples satisfying conditions {\it (C1)} and {\it (C3)}. 
 \begin{proposition}\label{exam-proxmap}
  Consider the functions
  \begin{equation}\label{exam-g}
  {\rm (i)} \  g(\cdot)=\lambda \|\cdot\|_q^q\ {\rm for}\ q=1/2\ {\rm or}\ 2/3, \  {\rm (ii)} \   g(\cdot)=\lambda \|\cdot\|_0+\lambda_0\|B\cdot\|_0 + \delta_{\Omega}(\cdot),
  \end{equation}
  where $B\in\mathbb{R}^{p\times n}$ and $\Omega$ is a box constraint containing $0$ as its interior point.  Let $\overline{x}$ be an accumulation point of the sequence generated by Algorithm \ref{hybrid}. 
  If $\overline{x} \in {\rm int}({\rm dom}g)$, then there exists $\varepsilon>0$ such that $\mathcal{R}_*$ is continuously differentiable, and the multifunction $\partial^2F_{*}$ is single-valued and locally Lipschitz continuous on $\mathbb{B}(\overline{x},\varepsilon)$. 
  \end{proposition}
  \begin{proof}
  From Proposition \ref{prop-supp} (ii), $\mathcal{R}_{\widetilde{L}^{-1}}$ is single-valued around $\overline{x}$. Then, the conclusion for (i) holds by the analytical expression of $\mathcal{P}_{\widetilde{L}^{-1} }g$, see \cite{xu12,xu23}.
  Now we consider $g(\cdot)= \lambda\|\cdot\|_0 + \lambda_0\|B\cdot\|_0 + \delta_{\Omega}(\cdot)$. From \cite[Lemma 8]{wu23b}, there exist $\varepsilon_0>0$ and $\nu> 0$ such that 
\begin{equation}\label{eq-uniform-lb}
    \min\{|Bx|_{\min}, |x|_{\min}\} \geq \nu, \ {\rm for} \ x\in \mathcal{T}_{*}(y) \ {\rm with } \ y\in \mathbb{B}(\overline{x},\varepsilon_0).
\end{equation}
Let $\mathcal{S}:\mathbb{R}^n \rightrightarrows \mathbb{R}^n$ be defined as $$\mathcal{S}(x) := \{ z\ | \ {\rm supp}(z) \subset {\rm supp}(x), {\rm supp}(Bz) \subset {\rm supp}(Bx)\}.$$
It is clear that for any $x\in \mathbb{R}^n$, $\mathcal{S}(x)$ is a linear subspace. 
 The proof is divided into two steps.

\noindent
 {\bf Step 1:}
 To prove that there exists $\varepsilon\in (0,\varepsilon_0)$ such that for every $x \in \mathcal{T}_{*}(y)$ with $y\in\mathbb{B}(\overline{x},\varepsilon)\cap \mathcal{O}$, $\mathcal{S}(x) = \mathcal{S}(\overline{x})$.
 We prove the result by contradiction. If the conclusion does not hold, there exist $\{x^k\}_{k\in\mathbb{N}}$ and $\{y^k\}_{k\in\mathbb{N}}$ such that $x^k \in \mathcal{T}_{*}(y^k)$ with $y^k \in \mathbb{B}(\overline{x}, 1/k)\cap \mathcal{O}$ and $\mathcal{S}(x^k) \neq \mathcal{S}(\overline{x})$ for all $k\in\mathbb{N}$. It is not hard to see that $\{x^k\}_{k\in\mathbb{N}}$ is bounded.
 If necessary taking a subsequence, we assume that $\lim_{k\rightarrow \infty} x^k = \widehat{x}$. It follows from \eqref{eq-uniform-lb} that $\min\{ |Bx^k|_{\min}, |x^k|_{\min}\} \geq \nu$ for all sufficiently large $k\in\mathbb{N}$, and $\min\{ |B\widehat{x}|_{\min}, |\widehat{x}|_{\min}\} \geq \nu$. Then, by applying Lemma \ref{prop-supp-Cx}, we have that $\mathcal{S}(x^k) = \mathcal{S}(\widehat{x})$ for sufficiently large $k$, resulting in $\mathcal{S}(\widehat{x}) \neq \mathcal{S}(\overline{x}).$ On the other hand, since $\lim_{k\rightarrow \infty} y^k = \overline{x}$, it follows from \cite[Theorem 1.25]{RW09} that $\widehat{x} \in \mathcal{T}_{*}(\overline{x})$. For the reason that $\mathcal{T}_{*}$ is single-valued at $\overline{x}$ (Proposition \ref{prop-supp} (ii)), we have $\widehat{x} = \overline{x}$, and this is a contradiction to $\mathcal{S}(\widehat{x}) \neq \mathcal{S}(\overline{x})$. The desired result holds.

 \noindent
 {\bf Step 2:} To prove the main result. 
 Pick any $y\in\mathbb{B}(\overline{x},\varepsilon)$. From Step 1, $\mathcal{S}(x) = \mathcal{S}(\overline{x}).$ Next we claim that $x = {\rm proj}_{\mathcal{S}(\overline{x})\cap \Omega}(y-\widetilde{L}^{-1}\nabla\!f(y))$. If not, there exists $\widetilde{x} \neq x$ such that $\widetilde{x} = {\rm proj}_{\mathcal{S}(\overline{x})\cap \Omega}(y-\widetilde{L}^{-1}\nabla\!f(y))$, which implies that 
   $\frac{1}{2}\|\widetilde{x} - (y-\widetilde{L}^{-1}\nabla\!f(y))\|^2 < \frac{1}{2}\|x - (y-\widetilde{L}^{-1}\nabla\!f(y))\|^2.$
 Note that $\widetilde{x}\in \mathcal{S}(\overline{x}) = \mathcal{S}(x)$, so $g(\widetilde{x})\le g(x)$. Combining the above two inequalities yields
 \[
 (1/2)\|\widetilde{x} - (y-\widetilde{L}^{-1}\nabla f(y))\|^2 + g(\widetilde{x}) < (1/2)\|x - (y-\widetilde{L}^{-1}\nabla f(y))\|^2 + g(x),
 \]
 a contradiction to the fact that $x\in \mathcal{T}_{*}(y)$. Thus, for every $y \in\mathbb{B}(\overline{x},\varepsilon)$, it holds that $\mathcal{T}_{*}(y) = {\rm proj}_{\mathcal{S}(\overline{x})\cap \Omega}(y-\widetilde{L}^{-1}\nabla f(y))$.  Note that $\mathcal{T}_*$ is locally Lipschitz continuous around $\overline{x}$.  By shrinking $\varepsilon$ if necessary, we assume that for any $x\in \mathcal{T}_{*}(y)$ with $y\in\mathbb{B}(\overline{x},\varepsilon)$, it holds that $x\in {\rm int}(\Omega)$. This leads to $\mathcal{T}_{*}(y) = {\rm proj}_{\mathcal{S}(\overline{x})}(y-\widetilde{L}^{-1}\nabla f(y))$ for $y\in \mathbb{B}(\overline{x},\varepsilon).$ Since the projection onto a subspace is a linear operator and $f$ satisfies Assumption \ref{ass-g} (i),
 we know that $\mathcal{R}_*$ is continuously differentiable with the Jacobian mapping being locally Lipschitz continuous on $\mathbb{B}(\overline{x},\varepsilon)$. Together with equation \eqref{GJac-Fgam} and Assumption \ref{ass-g} (i), there necessarily exists a neighborhood of $\overline{x}$ on which $\partial^2F_{*}$ is single-valued and locally Lipschitz continuous.
 \end{proof}

 \section{Numerical experiments}\label{sec6}
 
 In this section we apply Algorithm \ref{hybrid} (PGSSN) to solve the $\ell_q$-norm regularized problems \cite{xu12}, and fused zero-norm regularized problems \cite{wu23b}, respectively. All numerical tests are conducted on a laptop running in MATLAB R2024b on a 64-bit Windows System with an Intel(R) Core(TM) i7-13700H CPU 2.40 GHz and 64.0GB RAM.

\subsection{Implementation of PGSSN}
In PG step of PGSSN, we set $\gamma_{0,0} = 1$ and use the Barzilai-Borwein (BB) rule to compute the initial step-size $\gamma_{k,0}$ at $k$-th iteration in the following formula, 
$$ \gamma_{k,0} = \max \left\{ 10^{-20}, \min\left\{ 10^{20}, \frac{(x^k - x^{k-1})^{\top} (x^k - x^{k-1})}{(\nabla f(x^k) - \nabla f(x^{k-1}))^{\top}(x^k - x^{k-1})} \right\}\right\}.$$
We set $\gamma_k = \max\{ \gamma_{k,0}\eta^t , \frac{1}{L+\alpha}\},$ where $t$ is the smallest nonnegative integer such that
\begin{equation}\label{PG-condition}
    x^k \in \mathcal{T}_{\gamma_{k,0}\eta^{t}}(y^k) \ \ {\rm with}  \ \ F(x^k) \leq F_{\gamma_{k,0}\eta^{t}}(y^k) - \alpha \|x^k - y^k\|^2.
\end{equation}
If $\gamma_k = \gamma_{k,0}$, the initial step size determined by the BB rule can be too small, and we turn to conduct one more line search to enlarge its step size. In this case, we search for a largest nonnagative integer $t$ such that $x^k \in \mathcal{T}_{2^{t}\gamma_{k,0}}(y^k)$ and \eqref{PG-condition} is met. Then, we set $\gamma_k = \min\{2^{t}\gamma_{k,0},10^{20}\}$.

In our tested examples, following the proof method in Proposition \ref{exam-proxmap}, we can prove that the single-valuedness of $\mathcal{T}_{\widetilde{\gamma}_k}$ at a point implies its local Lipschitz continuity. Therefore we only judge whether $\mathcal{T}_{\widetilde{\gamma}_k}$ is single-valued in each iteration.

\subsection{$\ell_q$-norm regularized problems}
In this part, we conduct experiments for solving problem \eqref{model} with $g = \lambda\|\cdot\|_q^q$ ($q=1/2$) and $f = f_1$ or $f_2$, where $f_1 := \|A\cdot -b\|^2$ and $f_2 :=\sum_{i=1}^m\log(1+\exp(-b_iA_{i\cdot}\cdot))$.
We compare PGSSN with two existing second order methods. The first one is the BasGSSN (Basic Globalized ${\rm Semismooth}^*$ Newton method), proposed in \cite{gfrerer2024globally}, which can be regarded as a hybrid of PG and ${\rm semismooth}^*$ Newton method. In particular, BasGSSN uses the hybrid framework of ${\rm PANOC}^+$\cite{de2022proximal}, and alternates the PG step and the ${\rm semismooth}^*$ Newton step. The second solver we compare with is PGSSN with Newton step being replaced by the subspace regularized Newton method, and we call it by ${\rm HpgSRN}^{+}$, a modified version of HpgSRN proposed in \cite{wu23}, which solves $0\in \partial F(x)$ in stead of $0\in \mathcal{R}_{\widetilde{\gamma}_k}(x)$.
For these three solvers, we adopt the same randomly picked initial point, and use the same termination condition $ \|\mathcal{R}_{\gamma_k}(y^k)\| \leq 10^{-5}$. We adopt \cite[Algorithm 2]{gfrerer2024globally} to solve problem in step (2b) in Algorithm \ref{hybrid}. 

The data $(A,b)$ we used is from LIBSVM database (see \url{https://www.csie.ntu.edu.tw}). As suggested in \cite{huang10}, for space$\_$ga and triazines, we
expand their original features with polynomial basis functions. We solve \eqref{model} with $\lambda = \lambda_c\|A^\top b\|_{\infty}$ for two different $\lambda_c$'s for these three solvers.  For each test, we record the sparsity of the output (Nnz), the objective value of the output (Obj), the CPU time (Time) and the total number of the iterations (Iter).
 
The numerical comparisons are summarized in Tables \ref{tab1} and \ref{tab2}. In terms of efficiency, with the exception of triazines4, PGSSN consistently achieves the lowest or near-lowest CPU time. Notably, for the example log1p.E2006.train, the CPU time of PGSSN is approximately 70$\%$ of those of ${\rm HpgSRN}^+$ and BasGSSN. However, in triazines4, PGSSN does not perform well in terms of CPU time, while it achieves the lowest function value for two different $\lambda_c.$ For the Logistic regression, PGSSN uses the lowest or near-lowest time among the three solvers for rcv1 and news20 across both $\lambda_c$ values.

Regarding the objective value of the results, the three solvers demonstrate overall comparable performance. However, in certain cases, such as log1p.E2006.train with $\lambda_c = 10^{-4}$, PGSSN yields a solution with the lowest sparsity and the lowest objective value. This suggests that PGSSN can sometimes outperform the other two solvers in terms of the quality of the output objective value.

\begin{table}[!ht]
	\centering
        \scriptsize
	\caption{Numerical comparison on $\ell_q$-norm linear regression with LIBSVM datasets.\label{tab1}}
	\begin{tabular}{lcccccc}
		\toprule
		\begin{tabular}[l]{@{}l@{}}Data\\ $(m,n)$ \end{tabular}  & $\lambda_c$ & Solvers & Nnz& Obj & Time & Iter   \\
  \midrule
		\multirow{6}{*}{\begin{tabular}[l]{@{}l@{}}log1p.E2006.test\\ $(3308,1771946)$ \end{tabular}} 
   & \multirow{3}{*}{$10^{-4}$}  
   & PGSSN & 5 & 2.2334e2 & 425.16 & 97       \\
   & & ${\rm HpgSRN}^+$ & 5 & 2.3704e2 & 628.48 & 173      \\
   & & BasGSSN & 7 & 2.3764e2 & 193.88 & 82   \\

   \cline{2-7}

   & \multirow{3}{*}{$10^{-5}$}       
   & PGSSN & 487 & 1.6306e2 & 1106.66 & 245   \\
   & & ${\rm HpgSRN}^+$ & 511 & 1.6446e2 & 1261.91 & 262     \\
   & & BasGSSN & 484 & 1.7084e2 & 1338.15 & 202     \\

        \midrule
		\multirow{6}{*}{\begin{tabular}[l]{@{}l@{}}log1p.E2006.train\\ $(16087,4265669)$ \end{tabular}}  
   & \multirow{3}{*}{$10^{-4}$}  
   & PGSSN & 4 & 1.1599e3 & 1169.94 & 101       \\
   & & ${\rm HpgSRN}^+$ & 5 & 1.1613e3 & 1728.19 & 158      \\
   & & BasGSSN & 5 & 1.1605e3 & 1590.16 & 87   \\

   \cline{2-7}

   & \multirow{3}{*}{$10^{-5}$}       
   & PGSSN & 206 & 1.0282e3 & 2503.18 & 161   \\
   & & ${\rm HpgSRN}^+$ & 183 & 1.0229e3 & 4288.19 & 287     \\
   & & BasGSSN & 184 & 1.0379e3 & 3674.90 & 152     \\

    \midrule
        
		\multirow{6}{*}{\begin{tabular}[l]{@{}l@{}}space$\_$ga9\\ $(3107,5505)$ \end{tabular}} 
   & \multirow{3}{*}{$10^{-3}$}  
   & PGSSN & 7 & 3.6635e1 & 7.04 & 27       \\
   & & ${\rm HpgSRN}^+$ & 7 & 3.6972e1 & 7.81 & 58      \\
   & & BasGSSN & 7 & 3.7678e1 & 6.33 & 25   \\

   \cline{2-7}

   & \multirow{3}{*}{$10^{-4}$}       
   & PGSSN & 21 & 2.0928e1 & 11.25 & 35   \\
   & & ${\rm HpgSRN}^+$ & 19 & 2.1477e1 & 14.48 & 80     \\
   & & BasGSSN & 19 & 2.1522e1 & 18.61 & 43     \\
   
   \midrule
   
		\multirow{6}{*}{\begin{tabular}[l]{@{}l@{}}triazines4\\ $(186,557845)$ \end{tabular}}  
   & \multirow{3}{*}{$10^{-3}$}  
   & PGSSN & 21 & 1.7421 & 5545.45 & 127       \\
   & & ${\rm HpgSRN}^+$ & 17 & 1.7918 & 3735.47 & 224      \\
   & & BasGSSN & 21 & 1.9685 & 3683.55 & 85   \\

   \cline{2-7}

   & \multirow{3}{*}{$10^{-4}$}       
   & PGSSN & 75 & 5.7595e-1 & 23739.84 & 156   \\
   & & ${\rm HpgSRN}^+$ & 75 & 5.9447e-1 & 13152.14 & 255     \\
   & & BasGSSN & 74 & 5.9063e-1 & 16710.89 & 128     \\
		\bottomrule
	\end{tabular}
\end{table}
\begin{table}[!ht]
	\centering
        \scriptsize
	\caption{Numerical comparison on $\ell_q$-norm Logistic regressions with LIBSVM datasets.\label{tab2}}
	\begin{tabular}{lcccccc}
		\toprule
		\begin{tabular}[l]{@{}l@{}}Data\\ $(m,n)$ \end{tabular} &  $\lambda_c$ & Solvers & Nnz& Obj & Time & Iter \\

        \midrule
		\multirow{6}{*}{\begin{tabular}[l]{@{}l@{}}rcv1\\ $(20242,47236)$ \end{tabular}}  
        & \multirow{3}{*}{$10^{-3}$}  
        & PGSSN & 44 & 91.2285 & 1.252 & 41       \\
        & & ${\rm HpgSRN}^+$ & 45 & 91.0731 & 1.429 & 43      \\
        & & BasGSSN & 65 & 123.6691 & 2.030 & 39   \\

   \cline{2-7}

   & \multirow{3}{*}{$10^{-4}$}       
   & PGSSN & 53 & 11.3970 & 1.415 & 36   \\
   & & ${\rm HpgSRN}^+$ & 55 & 11.5112 & 1.370 & 51     \\
   & & BasGSSN & 71 & 15.0026 & 2.955 & 46     \\

    \midrule
        
		\multirow{6}{*}{\begin{tabular}[l]{@{}l@{}}news20\\ $(19996,1355191)$ \end{tabular}} 
        & \multirow{3}{*}{$10^{-3}$}  
        & PGSSN & 60 & 51.6978 & 13.225 & 40       \\
        & & ${\rm HpgSRN}^+$ & 53 & 51.2514 & 13.708 & 43      \\
        & & BasGSSN & 61 & 52.3313 & 17.945 & 33   \\

   \cline{2-7}

   & \multirow{3}{*}{$10^{-4}$}       
   & PGSSN & 61 & 24.8455 & 11.320 & 36   \\
   & & ${\rm HpgSRN}^+$ & 62 & 24.8602 & 19.648 & 58     \\
   & & BasGSSN & 68 & 25.4644 & 22.605 & 38     \\
		\bottomrule
	\end{tabular}
\end{table}

\subsection{Fused zero-norm regularized problems}
This subsection conducts numerical experiments on fused zero-norm regularized problems. We consider two models, first of which is \eqref{model} with $f = \|A\cdot - b\|^2$ and $g=g_1:=\lambda_0 \|B\cdot\|_0 + \lambda\|\cdot\|_0 + \delta_{\Omega}(\cdot)$,
where $B\in\mathbb{R}^{(n-1)\times n}$ satisfies $Bx = (x_1-x_2,..., x_{n-1} - x_n)$ for every $x\in\mathbb{R}^n$, and $\Omega$ is a box constraint containing the origin. The other is the fused zero norm plus $\ell_q$-norm regularization problems,
i.e., $f = \|A\cdot - b\|^2$ and $g=g_2:= \lambda_0\|B\cdot\|_0 + \lambda\|\cdot\|_q^q$
with $q = 1/2.$ The proximal mapping of $g_1$ can be achieved by a polynomial-time algorithm \cite{wu23b}, and the method for solving the proximal mapping of $g_2$ can be found in the Appendix.  We use PGSSN and PGiPN proposed in \cite{wu23b} to solve the first model, and the PGSSN to solve the second model, and compare their performance in recovery of blurred images. In particular, PGiPN is a hybrid of PG and projected regularized Newton method.

 Let $\overline{x} \in \mathbb{R}^n$ be a vector obtained by vectorizing a $256\times 256$ image ``cameraman.tif'', and be scaled such that all the entries belong to $[0,1]$. Let $A \in \mathbb{R}^{m\times n}$ be a matrix representing a Gaussian blur operator with standard deviation $4$ and a filter size of $9$, and the vector $b\in\mathbb{R}^m$ represent a blurred image obtained by adding Gauss noise $e \sim \mathcal{N}(0,\epsilon)$ with $\epsilon >0$ to $A\overline{x}$, i.e., $b= A\overline{x} + e$, respectively. 

  We denote by PGSSN0 for the case where PGSSN solves model with $g=g_1$, and PGSSNq as the case that PGSSN solves the model with $g=g_2$. We set $\lambda_1 = \lambda_2 = 0.005 \times \|A^{\top}b\|_{\infty}$. Under different $\epsilon$, we compare the performance of these three cases in terms of sparsity and structured sparsity of the outputs (Nnz and BxNnz, respectively), CPU time, number of iterations and the highest peak signal-to-noise ratio (PSNR), where PSNR$:= 10 \log_{10}(\frac{n}{\|\overline{x} - x^*\|})$. This feature is to measure the quality of the recovery. See Table \ref{tab4}. 

Based on Table \ref{tab4}, the performance of the two algorithms, PGSSN0 and PGiPN, shows similarity across various $\epsilon$ values. While for PGSSNq, although it takes more time to find a solution, the PSNR's of the returned images are slightly higher than those of the other two.

\begin{table}[!ht]
    \scriptsize
    \centering
    \caption{Numerical comparison on fused zero-norm regularization problems in recovery of blurred images.\label{tab4}}
    \begin{tabular}{lcccccccc}
        \toprule
        Noise Level & Solvers & Time & Nnz & BxNnz & Iter & PSNR  \\
        \midrule

        \multirow{3}{*}{\begin{tabular}[c]{@{}c@{}}$\epsilon=0.01$\end{tabular}} 
        & PGiPN & 450.63 &  63742 & 6465 & 65 & 25.307  \\
        & PGSSN0 & 401.23  &  63722 & 6385 & 57 & 25.379 \\
        & PGSSNq & 737.92  &  65163 & 6087 & 58 & 25.490  \\
        \midrule

        \multirow{3}{*}{\begin{tabular}[c]{@{}c@{}}$\epsilon=0.02$\end{tabular}} 
        & PGiPN & 603.85 &  63412 & 6651 & 89 & 24.538  \\
        & PGSSN0 & 508.80  &  63407 & 6598 & 67 & 24.612 \\
        & PGSSNq & 1094.31  &  65134 & 6397 & 77 & 24.628  \\
        \midrule

        \multirow{3}{*}{\begin{tabular}[c]{@{}c@{}}$\epsilon=0.03$\end{tabular}} 
        & PGiPN & 429.42 &  62929 & 7239 & 68 & 23.443 \\
        & PGSSN0 & 438.96 &  62894 & 7188 & 58 & 23.553  \\
        & PGSSNq & 1801.39 &  65015 & 6951 & 67 & 23.744  \\
        \midrule

        \multirow{3}{*}{\begin{tabular}[c]{@{}c@{}}$\epsilon=0.04$\end{tabular}} 
        & PGiPN & 1586.82 &  62040 & 7813 & 67 & 22.873  \\
        & PGSSN0 & 1415.73  & 62004 & 7754 & 53 & 22.970  \\
        & PGSSNq & 2224.64  & 64918 & 7628 & 74 & 23.053  \\
        \midrule

        \multirow{3}{*}{\begin{tabular}[c]{@{}c@{}}$\epsilon=0.05$\end{tabular}} 
        & PGiPN & 1652.84  &  61388 & 8452 & 72 & 22.320  \\
        & PGSSN0 & 1753.65  &  61362 & 8396 & 64 & 22.396 \\
        & PGSSNq & 1901.62 &  64806 & 8234 & 77 & 22.424  \\
        \bottomrule
    \end{tabular}
\end{table}

\section{Conclusions}
 For the nonconvex and nonsmooth composite problem \eqref{model}, we proposed a globalized semismooth Newton method. By exploiting the single-valuedness and local Lipschitz continuity of the proximal residual function, we used the semismooth Newton method to solve the system. The convergence of the  sequence to an $L$-stationary point was achieved by virtue of the KL exponent belong to $[\frac{1}{2}, \frac{1}{2-\varrho}]$ for $\varrho \in (0,1)$. If in addition the limit point is assumed to be a local minimum of $F$, at which the proximal residual function is metrically subregular, then the distance of the iterates from the set of $L$-stationary point is shown to be superlinearly convergent under mild conditions. The preliminary numerical experiments showed that the efficiency of our proposed algorithm is comparable to that of GSSN proposed in \cite{gfrerer2024globally}.  
 

\bibliographystyle{siam}
\bibliography{references}

\noindent
{\bf \LARGE Appendix}
\appendix

\section{Proximal mapping of $g$}\label{AppendixB}
In this section, we focus on the computation of proximal mappings of several kinds of $g$ which involves fused zero-norm, and discuss how to verify that its proximal mapping is single-valued.

\subsection{Fused zero-norms regularization}\label{Appendix-proxfzns}
In this part, we consider the case where $g= \lambda_0 \|B\cdot\|_0 + \lambda\|\cdot\|_0 + \delta_{\Omega}(\cdot)$. In what follows, we denote $\Omega = [l, u]$ for some $l, u\in\mathbb{R}^n$ with $l<0<u$.
A polynomial-time algorithm for seeking elements in the proximal mapping of $g$ was proposed in \cite{wu23b}. Although its proximal mapping is computable, to the best of our knowledge, the close-form expression of $\mathcal{P}_{\gamma}g$ is not available, so that we can not directly judge whether it is single-valued or not. Making a slight modification on the algorithm proposed in \cite{wu23b}, we provide a new algorithm for seeking the elements of the proximal mapping, which is able to verify whether $\mathcal{P}_{\gamma}g$ is single-valued.  

Assume that $n \geq 2$. For any given $z\in\mathbb{R}^n$, the proximal mapping of $\mathcal{P}_1g$ at $z$ is the set of optimal solutions to the following problem, 
	\begin{equation}\label{fusedl0-prox}
		\min_{x\in\mathbb{R}^n}\, h(x;z):= \frac{1}{2}\|x-z\|^2+\lambda_0 \|Bx\|_0 + \lambda \|x\|_0 + \delta_{\Omega}(x).
	\end{equation}
	 To simplify the deduction, for each $i\in [n]$, we define $\omega_i:\mathbb{R}\rightarrow \overline{\mathbb{R}}$ by $\omega_i(\alpha) := \lambda |\alpha|_0 + \delta_{[l_i,u_i]}(\alpha)$. It is clear that for all $x\in\mathbb{R}^n$, $\lambda \|x\|_0 +\delta_{\Omega}(x)= \sum_{i=1}^n \omega_i(x_i)$. 
     Let $H(0):= -\lambda_0$ and $B_{[0][1]}:= 0$, and for each $s\in [n]$, define 
 \begin{equation}\label{Hfun}
 \begin{aligned}
         & H(s)\! := \!\min_{y\in\mathbb{R}^s}  h_s(y;z_{1:s}), \\
         & {\rm where} \  h_s(y;z_{1:s})\!:=\!\frac{1}{2}\|y- z_{1:s}\|^2 + \lambda_0 \|B_{[s-1] [s]}y\|_0 +  \sum_{j=1}^s \omega_j(y_j) .
 \end{aligned}
 \end{equation}
 It is immediate to see that $H(n)$ is the optimal value to problem \eqref{fusedl0-prox}. 
    For each $s\in[n]$, define function $P_{s}\!:[0\!:\!s\!-\!1]\times \mathbb{R} \to \overline{\mathbb{R}}$ by 
	\begin{align}\label{Ps-fun}
		P_s(i, \alpha) := H(i) + \frac{1}{2}  \|\alpha {\bf 1} - z_{i+1:s}\|^2 + \sum_{j=i+1}^s \omega_j(y_j)  +\lambda_0.
	\end{align}
    The following lemma is a modified version of \cite[Lemma 9]{wu23b}, which provides a line to search for the elements of the proximal mapping of $g$.
    \begin{lemma}\label{prop-transform}
     Fix any $s\in [n]$. The following statements are true.
     \begin{itemize}
         \item[{\rm (i)}] $H(s) = \min_{i\in[0:s-1], \alpha \in \mathbb{R}} P_s(i,\alpha)$.
         \item[{\rm(ii)}] $(i_s^*, \alpha_s^*) \in \mathop{\arg\min}_{i\in[0:s-1], \alpha \in \mathbb{R}} P_s(i,\alpha)$ if and only if $y^* = (y^*_{1:i_s^*}; \alpha_s^* {\bf 1})$ is a global solution of $\min_{y\in\mathbb{R}^s} h_{s}(y; z_{1:s})$ with $y_{1:i_s^*}^*\in \mathop{\arg\min}_{v\in\mathbb{R}^{i_s^*}} h_{i_s^*}(v;z_{1:i_s^*})$ and $y^*_{i_s^*} \neq \alpha_s^*$ if $i_s^* \neq 0$.
     \end{itemize}
    \end{lemma}
    \begin{proof}
    The proof of (i) can be found in \cite[Lemma 9]{wu23b}. Next, we prove (ii). 

    \noindent
    $\Longrightarrow:$
    If $i_s^* \neq 0$, by part (i) and the definitions of $\alpha_s^*$ and $i_s^*$, 
        \begin{align*}
            H(s) & = \min_{i\in[0:s-1],\alpha \in \mathbb{R}} P_s(i,\alpha) = H(i_s^*) + \frac{1}{2}  \|\alpha_s^*{\bf 1} - z_{i_s^*+1:s}\|^2 + \sum_{j=i_s^*+1}^s\omega_j(\alpha_s^*) +\lambda_0 \\
            & = h_{i^*_s}(y^*_{1:i_s^*};z_{1:i_s^*})  + \frac{1}{2}  \|y^*_{i_s^*+1:s} - z_{i_s^*+1:s}\|^2 + \sum_{j=i_s^*+1}^s\omega_j(y^*_{j}) +\lambda_0 \geq h_{s}(y^*; z_{1:s}),
        \end{align*}
        where the last inequality follows by the definition of $h_s(\cdot, z_{1:s})$. From the definition of $H(s)$, we must have $H(s) = h_s(y^*;z_{1:s})$, which together with the above inequality implies that $y^*_{i_s^*} \neq \alpha_s^*.$
        If $i_s^* = 0$, 
        $$H(s) = \min_{i\in[0:s-1],\alpha \in \mathbb{R}} P_s(i,\alpha) = H(0) + \frac{1}{2}  \|y^* - z_{1:s}\|^2 + \sum_{j=1}^s\omega_j(y^*_{j}) +\lambda_0 = h_s(y^*;z_{1:s}),$$
        Therefore, $H(s) \geq h_s(y^*;z_{1:s}).$ Along with the definition of $H(s)$, $H(s) = h_s(y^*;z_{1:s})$. 

      \noindent
      $\Longleftarrow$: The case $i_s^* = 0$ is clear. Now we consider $i_s^*\neq 0$. From (i) we have 
    \begin{align*}
         \min_{i\in[0:s-1], \alpha \in \mathbb{R}} P_s(i,\alpha) & = H(s) = h_s(y^*; \alpha_s^* {\bf 1}) \\
        & = h_{i_s^*}(y^*_{1:i_s^*};z_{1:i_s^*}) + \lambda_0 + \frac{1}{2}\|\alpha_s^*{\bf 1} - z_{i_s^*+1:s}\|^2 + \sum_{j=i_s^*+1}^s\omega_j(\alpha_s^*) \\
        & = H(i_s^*) + \lambda_0 + \frac{1}{2}\|\alpha_s^*{\bf 1} - z_{i_s^*+1:s}\|^2 + \sum_{j=i_s^*+1}^s\omega_j(\alpha_s^*), 
    \end{align*}
    where the third equality follows by the definition of $h_s$ and $y^*_{i_s^*} \neq \alpha_s^*.$ The above equation implies that $(i_s^*,\alpha_*)$ is a global minimum of $P_s(i,\alpha).$
    \end{proof}
  
   By virtue of this lemma, problem \eqref{fusedl0-prox} can be recast as a mixed integer programming with the objective function given in \eqref{Ps-fun}. As pointed out in \cite{wu23b}, Lemma \ref{prop-transform} suggests a recursive method to obtain an optimal solution to problem \eqref{fusedl0-prox}. An outline is present as follows.
   \begin{numcases}{}
 {\rm Set\ the\ current\ changepoint}\ s = n. \notag \\{\bf\rm  While} \ s > 0 \  {\bf\rm do} \notag\\ \quad \quad {\rm Find} \ (i_{s}^*, \alpha_{s}^*) \in \mathop{\arg\min}_{i\in [0:s-1], \alpha\in\mathbb{R}} P_{s}(i,\alpha).\notag\\ 
 \quad \quad {\rm Let} \ x^*_{i_s^*+1:s} = \alpha_{s}^*{\bf 1} \ {\rm and} \ s \gets i_{s}^*.\notag\\   {\rm End}\notag
 \end{numcases}
 From Lemma \ref{prop-transform} it is not hard to see that $x^*$ obtained from the above procedure is an optimal solution to \eqref{fusedl0-prox}. Then, if $(i_s^*, \alpha_s^*)$ is unique for every $s$ involved in the while loop, the optimal solution to \eqref{fusedl0-prox} is also unique. To give a method to solve $\min_{i\in[0:s-1],\alpha \in \mathbb{R}}P_s(i,\alpha)$, which can also judge whether the optimal solution set of problem \eqref{fusedl0-prox} is a singleton, we provide some preparations.  

\begin{proposition}\label{prop-Ps}
		For each $s\in [n]$, let $P_{s}^*(\alpha):= \min_{i\in[0:s-1]}P_s(i,\alpha)$. 
		\begin{itemize}
			\item[{\rm (i)}]  For all $\alpha\in\mathbb{R}$, if $s=  1$, $P_s^*(\alpha) =  \frac{1}{2}(\alpha-z_1)^2 + \omega_1(\alpha)$; and if $s \in [2:n],$ $$P_s^*(\alpha) = \min\Big\{ P_{s-1}^*(\alpha), \min_{\alpha'\in\mathbb{R}} P_{s-1}^*(\alpha') \!+\! \lambda_0\Big\} \!+\! \frac{1}{2}(\alpha\! -\! z_s)^2\! +\! \omega_s(\alpha).$$   
 \item[{\rm (ii)}] Let $\mathcal{R}_1^0\!:=\mathbb{R}$, and $\mathcal{R}_s^{i}\!:=\mathcal{R}_{s-1}^{i} \cap \mathcal{V}_s^{s-1}$
			for all $s\in[2\!:\!n]$ and $i\in[0\!:\!s\!-\!2]$, where 
\begin{align*}
    & \mathcal{R}_s^{s-1}:= \left\{\alpha\in\mathbb{R} \ | \ P_{s-1}^*(\alpha) \geq  \min_{\alpha'\in\mathbb{R}} P_{s-1}^*(\alpha') +\lambda_0 \right\}, \\
    & \mathcal{V}_s^{s-1}:= \left\{\alpha\in\mathbb{R} \ | \ P_{s-1}^*(\alpha) \leq \min_{\alpha'\in\mathbb{R}} P_{s-1}^*(\alpha') +\lambda_0 \right\}.
\end{align*}  
     Then, for each $s\in[n]$ and $i\in [0\!:\!s\!-\!1]$,  $P_s^*(\alpha) = P_s(i,\alpha)$ if and only if $\alpha \in \mathcal{R}_s^{i}$.
		\end{itemize} 
	\end{proposition}
	\begin{proof}
		 (i) See \cite[Proposition 10]{wu23b}.
		
		\noindent
		 (ii) $\Longleftarrow$: Since for any $\alpha \in \mathbb{R} = \mathcal{R}_1^0$, $P_1^*(\alpha) = P_1(0,\alpha)$, the result holds for $s= 1$. For $s \in [2\!:\!n]$, by the definitions of $\mathcal{R}_s^{s-1}$ and part (i), we have for all $\alpha \in \mathcal{R}_s^{s-1}$, 
		\[
			P_s^*(\alpha)  = \min_{\alpha'\in\mathbb{R}} P_{s-1}^*(\alpha') + \lambda_0 + \frac{1}{2}(\alpha - z_s)^2 + \omega_s(\alpha)  = P_s(s-1,\alpha).
		\]
		  Next we consider $s\in[2\!:\!n]$ and $i \in [0\!:\!s\!-\!2]$.
		We argue by induction that $P_s^*(\alpha) = P_s(i,\alpha)$ when $\alpha \in \mathcal{R}_s^{i}$. Indeed, when $s=2$, since $\mathcal{R}_2^0 = \mathcal{R}_1^0\cap \mathcal{V}_2^1 = \mathcal{V}_2^1$, we conclude that for any $\alpha\in\mathcal{R}_2^0$, $P_{1}^*(\alpha)\leq \min_{\alpha'\in\mathbb{R}}P_1^*(\alpha')+\lambda_0$, which by part (i) implies that $P_2^*(\alpha) =P_{1}^*(\alpha)+\frac{1}{2}(\alpha-z_2)^2+\omega_2(\alpha)=P_{1}(0,\alpha)+\frac{1}{2}(\alpha-z_2)^2+\omega_2(\alpha)=P_2(0,\alpha)$. Assume that the result holds when $s=j$ for some $j\in [2\!:\!n\!-\!1]$. We consider the case for $s=j\!+\!1$. For any $i\in [0\!:\!s\!-\!2]$, by definition, $\mathcal{R}_{s}^i = \mathcal{R}_{s-1}^i \cap \mathcal{V}_{s}^{s-1}$. Then, for any $\alpha \in \mathcal{R}_{s}^i$, 
		\begin{align*}
			P_{s}^*(\alpha) & = P_{s-1}^*(\alpha) + \frac{1}{2}(\alpha -z_{s})^2 + \omega_s(\alpha) = P_{s-1}(i,\alpha)+ \frac{1}{2}(\alpha -z_{s})^2 + \omega_s(\alpha)\\
			& = H(i)+ \frac{1}{2}\|\alpha {\bf 1} - z_{i+1:s-1}\|^2 +\sum_{j=i+1}^{s-1}\omega_j(\alpha) +\lambda_0  + \frac{1}{2}(\alpha\!-\!z_{s})^2 + \omega_s(\alpha) \\
            & = H(i) + \frac{1}{2}\|\alpha {\bf 1} - z_{i+1:s}\|^2 + \sum_{j=i+1}^{s}\omega_j(\alpha)+\lambda_0  = P_{s}(i,\alpha),
		\end{align*}
        where the first equality uses $\alpha \in \mathcal{V}_s^{s-1}$, and the second equality is using $P_{s-1}^*(\alpha)=P_{s-1}(i,\alpha)$ for $\alpha \in \mathcal{R}_{s-1}^i$ implied by induction. Hence, the conclusion holds for $s=j+1$ and any $i\in[0\!:\!s\!-\!2]$. By induction, we obtain the desired result.

        \medskip
        \noindent
        $\Longrightarrow:$ 
        We prove by induction. By definition, $\mathcal{R}_1^0 = \mathbb{R}$. It is clear that the result holds for $s= 1$. Next, we consider $s\in [2\!:\!n]$. 
        If $P_s^*(\alpha)=P_s(s-1,\alpha)$ hold, we have
        $$ P_s^*(\alpha) = H(s-1) + \frac{1}{2}(\alpha - z_s)^2 + \omega_s(\alpha) + \lambda_0 = \min_{\alpha'\in\mathbb{R}}P_{s-1}^*(\alpha') + \frac{1}{2}(\alpha - z_s)^2 + \omega_s(\alpha) + \lambda_0,$$
        which by part (i) implies that $ P_{s-1}^*(\alpha) \geq \min_{\alpha'\in\mathbb{R}} P_{s-1}^*(\alpha') + \lambda_0$, and hence
        $\alpha \in \mathcal{R}_{s}^{s-1}$. Next we consider the case where $s\in [2:n] $ and $ i\in [0\!:\!s\!-\!2]$. Indeed, $P_s^*(\alpha)=P_s(i,\alpha)$ implies
        $P_s(i,\alpha) \leq P_s(s-1,\alpha)$, which by definition yields
        $$ H(i) + \frac{1}{2}\|\alpha {\bf 1} - z_{i+1:s}\|^2 + \lambda_0 + \sum_{j=i+1}^s\omega_j(\alpha) \leq \min_{\alpha'\in\mathbb{R}} P_{s-1}^*(\alpha') + \lambda_0 + \frac{1}{2}(\alpha-z_s)^2 + \omega_s(\alpha).  $$
        Subtracting $\frac{1}{2}(\alpha-z_s)^2+\omega_s(\alpha)$ on both sides yields 
        $$ P_{s-1}^*(\alpha) \leq P_{s-1}(i,\alpha) \leq \min_{\alpha'\in\mathbb{R}} P_{s-1}^*(\alpha') + \lambda_0,$$
        where the first inequality uses the definition of $P_{s-1}^*$,
        and hence $\alpha \in \mathcal{V}_{s}^{s-1}$. Moreover, it follows from $P_s^*(\alpha)=P_s(i,\alpha)$ that 
        $$H(i) +  \frac{1}{2}\|\alpha {\bf 1} - z_{i+1:s}\|^2 + \lambda_0 + \sum_{j=i+1}^s\omega_j(\alpha) = P_s(i,\alpha) = P_{s-1}^*(\alpha) + \frac{1}{2}(\alpha - z_s)^2 + \omega_s(\alpha),$$
     which implies that $P_{s-1}^*(\alpha) = P_{s-1}^i(\alpha)$, and this together with the result in induction yields $\alpha \in \mathcal{R}_{s-1}^i$. Therefore, we have $\alpha \in \mathcal{R}_{s-1}^i \cap \mathcal{V}_{s}^{s-1} = \mathcal{R}_s^i.$ By induction, we obtain the desired result.
 \end{proof}

 By using this proposition, we present the algorithm for seeking the elements of the proximal mapping of $g$ in Algorithm \ref{pruning}. By using Proposition \ref{prop-Ps} (ii), we obtain the following equivalent characterization of the single-valuedness of $\mathcal{P}_{1}g$.
 \begin{proposition}
     The proximal mapping of $g$ is a singleton if and only if $\alpha_s^* = \mathop{\arg\min}_{\alpha \in \mathbb{R}} P_{s}^*(\alpha)$, and  $i_s^* = \left\{i\ \lvert\ \alpha_s^* \in \mathcal{R}_{s}^{i} \right\}$ for every $s$ involved in the while loop of Algorithm \ref{pruning}.
 \end{proposition}
 \begin{proof}
 The equivalence directly follows by Proposition \ref{prop-transform} (ii).
 \end{proof}

 	\begin{algorithm}[!h]
		\caption{ (Computing ${\rm prox}_{\lambda_1 \|B\cdot\|_0 + \omega(\cdot)}(z)$) }\label{pruning}  
		1. \textbf{Initialize:} Compute $P_1^*(\alpha)= \frac{1}{2}(z_1-\alpha)^2 + \omega_1(\alpha)$ and set $\mathcal{R}_1^0 = \mathbb{R}.$ 
		
		2. {\bf For} $s= 2,\ldots,n\ $  {\bf do}
		
		3. \quad $P_s^*(\alpha):=\min\{ P_{s-1}^*(\alpha), \min_{\alpha'\in\mathbb{R}} P_{s-1}^*(\alpha') + \lambda_1\} +\frac{1}{2}(\alpha- z_s)^2 + \omega_s(\alpha).$
		
		4. \quad Compute $\mathcal{R}_{s}^{s-1}$ and $\mathcal{V}_{s}^{s-1}$.
		
		5. \quad {\bf For} \ $i = 0,\ldots, s-2$ {\bf do}
		
		6. \quad \quad $\mathcal{R}_s^{i} = \mathcal{R}_{s-1}^{i} \cap \mathcal{V}_{s}^{s-1}$.
		
		7. \quad {\bf End}
		
		8. {\bf End}
		
		
		9. Set the current changepoint $s = n$.
		
		10. {\bf While}  $s > 0$  {\bf do}
		
		11.\quad  Find $\alpha_s^* \in \mathop{\arg\min}_{\alpha \in \mathbb{R}} P_{s}^*(\alpha)$, and  $i_s^* \in \left\{i\ \lvert\ \alpha_s^* \in \mathcal{R}_{s}^{i} \right\}$.
		
		12. \quad $x^*_{i_s^*+1:s} = \alpha_s^*{\bf 1}$ and $s\gets i_s^*.$
		
		13. {\bf End}
	\end{algorithm}

\subsection{Proximal mapping of $\lambda_0 \|Bx\|_0 + \vartheta(x)$}
In this part, we consider the proximal mapping of $g:=\lambda_0 \|B\cdot\|_0 + \vartheta(\cdot)$, and aim at proving that if $\vartheta$ satisfies the following assumption, the proximal mapping of $g$ is accessible. The key concept underlying this section is the optimal partitioning of changepoints, as introduced in \cite{liebscher1999potts}. For clarity, we provide a brief overview of this idea below.

\begin{assumption}\label{ass-proxBx}
   {\rm (i)} $\vartheta$ is separable, i.e., $\vartheta(x) = \sum_{i=1}^n \vartheta_i(x_i).$ 
   
   \noindent
   {\rm (ii)} For each $j,k \in [n]$, the proximal mapping of $\sum_{i=j}^k \vartheta_i(\alpha)$ is accessible, i.e., for any $\beta \in \mathbb{R}^n$ and $\gamma >0$, one of the optimal solutions to $\min_{\alpha \in \mathbb{R}} \ \frac{1}{2\gamma}(\alpha-\beta)^2 + \sum_{i=j}^k \vartheta_i(\alpha)$ is computable.
\end{assumption}

For any given $z\in\mathbb{R}^n$, we write $h(x;z) = \frac{1}{2}\|x-z\|^2+ \lambda_0\|Bx\|_0 + \vartheta(x) .$ Let $H(0) = -\lambda_0$ and $B_{[0][1]}:= 0$ and for each $s\in [n]$, define 
\begin{equation}\label{Hfun_vartheta}
         H(s)\! := \!\min_{y\in\mathbb{R}^s}  h_s(y;z_{1:s}), \ {\rm where} \  h_s(y;z_{1:s})\!:=\!\frac{1}{2}\|y- z_{1:s}\|^2 + \lambda_0\|B_{[s-1] [s]}y\|_0 + \sum_{j=1}^s \vartheta_j(y_j).
 \end{equation}
 For each $s\in [n]$, define function $P_s:[0:s-1]\times \mathbb{R} \rightarrow \overline{\mathbb{R}}$ by
 $$ P_s(i,\alpha):= H(i) + \frac{1}{2}\|\alpha {\bf 1} - z_{i+1:s}\|^2 + \sum_{j=i+1}^s \vartheta_j(\alpha) + \lambda_0. $$
By following a similar deduction to that in Proposition \ref{prop-transform}, we have  
$$H(s) = \min_{i\in[s-1], \alpha \in \mathbb{R}} P_s(i,\alpha).$$
By Assumption \ref{ass-proxBx}, for any $s\in [n]$ and $i \in [0:s-1]$, one of the optimal solutions of $\min_{\alpha} P_s(i,\alpha)$ is accessible. 
Making use of this, we calculate $H(i)$ from $i=1$ to $n$ by the following formulas, 
\begin{align*}
    & H(1) = \min_{\alpha \in \mathbb{R}}P_1(0,\alpha);\\
    & H(2) = \min\{ \min_{\alpha \in \mathbb{R}}P_2(0,\alpha), \min_{\alpha \in \mathbb{R}}P_2(1,\alpha)\};\\
    & \ldots \\
    & H(n) = \min\{ \min_{\alpha \in \mathbb{R}}P_n(0,\alpha), \min_{\alpha \in \mathbb{R}}P_n(1,\alpha),\ldots, \min_{\alpha \in \mathbb{R}}P_n(n-1,\alpha)\}.
\end{align*}
Write $Q(i,s) := \min_{\alpha\in\mathbb{R}} \frac{1}{2}\|\alpha {\bf 1} - z_{i+1:s}\|^2 + \sum_{j=i+1}^s \vartheta_j(\alpha).$
Therefore, there exist $k_1,k_2,...,k_t \in [n]$ with $k_1 > k_2>...>k_t$ such that 
\begin{equation}\label{Hnexpress}
\begin{aligned}
    & H(n) = H(k_1) + Q(k_1+1,n) + \lambda,\\
    & H(k_1) = H(k_2) + Q(k_2+1, k_1) + \lambda,\\
    & \ldots \\
    & Q(k_t) = P(1,k_t), 
\end{aligned}
\end{equation}
which implies that 
$$ H(n) = Q(1,k_t)+ Q(k_{t}+1, k_{t-1}) + ...+ Q(k_1+1, n) + t\lambda.$$
By letting $x^* = (v^*_{1:k_t}, v^*_{k_{t}+1:k_{t-1}},\ldots , v^*_{k_1+1:n})$, where $v^*_{j:k} \in \mathop{\arg\min}_{\alpha\in\mathbb{R}} \frac{1}{2}\|\alpha {\bf 1} - z_{j:k}\|^2 + \sum_{i=j}^k \vartheta_i(\alpha)$, we have $h(x^*;z) = H(n)$. That is, $x^*$ is a global minimum of $h(\cdot; z)$.

Fix any $z\in \mathbb{R}^n$. It is clear that if the pair $\{k_1,...,k_t\}$ satisfying \eqref{Hnexpress} is unique, and the corresponding $v^*_{k_i:k_{i+1}}$ is unique for every $i\in [t-1]$, then the proximal mapping of $\lambda_0\|Bx\|_0 + \vartheta(x)$ is a singleton.

\end{document}